\documentclass[final]{siamltex}

\usepackage[nosumlimits]{amsmath}
\usepackage{amsfonts,amssymb}
\usepackage{graphicx}
\graphicspath{{./Figures/}}
\usepackage{color}
\usepackage{ulem}
\usepackage{multicol}
\usepackage{cite} 
\usepackage{placeins}
\usepackage{pifont}
\usepackage[dvipsnames]{xcolor}
\usepackage{tikz}
\usepackage{pgfplots}
\usetikzlibrary{shapes}
\usepackage[textwidth=20mm]{todonotes}
\usepackage{algorithm}
\usepackage{algpseudocode}
\usepackage{algorithmicx}
\usepackage{setspace}
\usepackage{url}
\usepackage{hyperref}

\DeclareMathAlphabet{\mathbf}{OT1}{cmr}{bx}{it}

\newcommand{\vb}{{\mathbf b}}

\newcommand{\ve}{{\mathbf e}}

\newcommand{\vf}{{\mathbf f}}

\newcommand{\vq}{{\mathbf q}}
\newcommand{\vr}{{\mathbf r}}

\newcommand{\vu}{{\mathbf u}}
\newcommand{\vv}{{\mathbf v}}

\newcommand{\vw}{{\mathbf w}}
\newcommand{\vx}{{\mathbf x}}
\newcommand{\vy}{{\mathbf y}}

\newcommand{\vnull}{\boldsymbol{0}}
\newcommand{\spK}{{\cal K}}

\renewcommand{\d}{\,\mathrm{d}}
\newcommand{\R}{\mathbb{R}}

\newcommand{\C}{\mathbb{C}}

\newcommand{\CN}{\mathbb{C}^N}

\newcommand{\CNN}{\mathbb{C}^{N \times N}}

\newcommand{\CNm}{\mathbb{C}^{N \times m}}

\newcommand{\real}{\mathfrak{R}}
\newcommand{\AHA}{A^H\!\!A}

\DeclareMathOperator{\Span}{span}

\DeclareMathOperator{\sign}{sign}

\newtheorem{example}[theorem]{Example}
\newcommand{\SHS}{S^H\!S}

\let\oldexample\example
\renewcommand{\example}{\oldexample\normalfont}
\newtheorem{remarksimple}[theorem]{Remark}
\let\oldremarksimple\remarksimple
\renewcommand{\remarksimple}{\oldremarksimple\normalfont}
\newenvironment{remark}{\begin{remarksimple}}{\hfill$\diamond$\end{remarksimple}}
\newtheorem{experiment}[theorem]{Experiment}
\let\oldexperiment\experiment
\renewcommand{\experiment}{\oldexperiment\normalfont}
\renewcommand{\real}{\mbox{Re}}

\usepackage[dvipsnames]{xcolor}
\usepackage{pgfplots}
\usepgfplotslibrary{groupplots}
\usetikzlibrary{external}

\pgfkeys{/pgf/images/include external/.code={\includegraphics{#1}}}

\usetikzlibrary{decorations.markings}
\makeatletter
\tikzset{
  nomorepostactions/.code={\let\tikz@postactions=\pgfutil@empty},
  mymark/.style 2 args={decoration={markings,
    mark= between positions 0 and 1 step (1/9)*\pgfdecoratedpathlength with{%
        \tikzset{#2,every mark}\tikz@options
        \pgfuseplotmark{#1}%
      },  
    },
    postaction={decorate},
    /pgfplots/legend image post style={
        mark=#1,mark options={#2},every path/.append style={nomorepostactions}
    },
  },
}
\makeatother

\pgfplotscreateplotcyclelist{list_std}{%
Cerulean,line width=1.5pt, solid\\
BurntOrange,line width=1.5pt, densely dashed\\
OliveGreen,line width=1.5pt, loosely dotted\\
Orchid,line width=1.5pt, dash dot\\
}

\pgfplotscreateplotcyclelist{list_no_lines}{%
Cerulean,line width=1.25pt, mark=square, only marks\\
BurntOrange,line width=1.25pt, mark=o, only marks\\
OliveGreen,line width=1.25pt, mark=diamond, only marks\\
Orchid,line width=1.25pt, mark=x, only marks\\
}
\pgfplotsset{compat=1.17}

\title{Randomized sketching for Krylov approximations \\ of large-scale matrix functions}
\author{Stefan G\"uttel\thanks{Department of Mathematics, The University of Manchester, M13\,9PL Manchester, United Kingdom, \texttt{stefan.guettel@manchester.ac.uk}.  S.\,G. acknowledges a fellowship from The Alan Turing Institute under the EPSRC grant EP/W001381/1.} \and Marcel Schweitzer\thanks{School of Mathematics and Natural Sciences, Bergische Universit\"at Wuppertal, 42097 Wuppertal, Germany, \texttt{marcel@uni-wuppertal.de}.}}
\date{\today}
\begin{document}

\newcommand{\rev}[1]{{\color{black}#1}}
\newcommand{\revv}[1]{{\color{black}#1}}

\renewcommand{\thefootnote}{\fnsymbol{footnote}}
\maketitle \pagestyle{myheadings} \thispagestyle{plain}
\markboth{S. G\"UTTEL AND M. SCHWEITZER}{RANDOMIZED SKETCHING OF MATRIX FUNCTIONS}

\begin{abstract}
The computation of $f(A)\vb$, the action of a matrix function on a vector, is a task arising in many areas of scientific computing. In many applications, the matrix $A$ is sparse but so large that only a rather small number of Krylov basis vectors can be stored. Here we discuss a new approach to overcome this limitation by randomized sketching combined with an integral representation of $f(A)\vb$. Two different approximation methods are introduced, one based on sketched FOM and another based on sketched GMRES. The convergence of the latter method is analyzed for Stieltjes functions of positive real matrices. We also derive a closed form expression for the sketched FOM approximant and bound its distance to the full FOM approximant. Numerical experiments demonstrate the potential of the presented sketching approaches.     
\end{abstract}




\section{Introduction}

The computation of $f(A)\vb$, the action of a function $f$ of $A\in\CNN$ on a vector $\vb\in\mathbb{C}^N$, is a task arising in many areas of scientific computing. By far the most popular  methods for this task are polynomial~\cite{DruskinKnizhnerman1989,Saad1992} and  rational~\cite{Guettel2013,GuettelKnizhnerman2013,vdEH06,DruskinKnizhnerman1998} Krylov methods. In many applications, the matrix $A$ is sparse but so large that only a rather small number of Krylov basis vectors of size~$N$ can be stored. Furthermore, for non-Hermitian matrices $A$, the arithmetic cost of orthogonalizing a Krylov basis can become overwhelming. This naturally limits the attainable accuracy of Krylov methods which perform full orthogonalization  and need to store at least one additional vector per iteration. Several approaches are available for overcoming the memory problem, including
\begin{itemize}
\item  two-pass Krylov methods for Hermitian $A$ as in~\cite{Borici2000,FrommerSimoncini2008}, which roughly double the computational effort,
\item methods based on a-priori rational approximation of $f$ on a spectral region of~$A$, followed by a (short recurrence) Krylov iteration for the resulting shifted linear systems of equations~\cite{VanDenEshofFrommerLippertSchillingVanDerVorst2002,FrommerMaass1999}, and 
\item restarted Krylov methods~\cite{AfanasjewEtAl2008a,EiermannErnst2006,FrommerGuettelSchweitzer2014a,FrommerGuettelSchweitzer2014b,IlicTurnerSimpson2010,Schweitzer2016thesis,TalEzer2007} which, similar to restarted methods for linear systems, construct a series of Krylov iterates in such a way that each ``cycle'' of the method only requires a fixed amount of storage and fixed cost for orthogonalization.
\end{itemize}
We also refer the reader to the recent survey~\cite{GuettelKressnerLund2020} covering limited-memory polynomial methods for the general $f(A)\vb$ problem, and more specifically to~\cite{GS21} for the case of Stieltjes functions of Hermitian matrices.

In this paper we discuss a new technique to overcome the issues of excessive memory requirements and orthogonalization cost in Krylov  methods for the $f(A)\vb$ problem. Our approach is based on the sketched \rev{Krylov}  approximation of the shifted linear systems $(t I + A)\vx(t) = \vb$ arising with the integral form
\[
f(A)\vb =  \int_\Gamma f(t) (t I + A)^{-1}\vb \d\mu(t) =  \int_\Gamma \vx(t) \d \mu(t).
\]
This representation exists for any function $f$ that is analytic on and inside a closed contour $\Gamma$ that encloses the negated spectrum $-\Lambda(A)$. In the case $\mathrm{d}\mu(t) = -(2\pi i)^{-1}\mathrm{d}z$ we obtain the Cauchy integral representation of $f(-A)\vb$; see~\cite[Def.~1.11]{Higham2008}. The above integral  representation also contains the important class of Stieltjes functions, in which case $\Gamma = [0,+\infty)$ and 
$\mu(t)$  is a monotonically increasing and nonnegative function on $\Gamma$ with $\int_\Gamma 1/(t+1)\d \mu(t)<\infty$; see, e.g., \cite{Henrici1977}. \rev{Important examples of Stieltjes functions are $f(z)=z^{-\alpha}$ for $\alpha\in (0,1)$ and $f(z) = \log(1+z)/z$. 
Some other interesting functions like $\widetilde f(z) = z^\alpha$ for $\alpha\in (0, 1)$, including the square
root, and $\widetilde f(z) = \log(1 + z)$ can be written as $\widetilde f(z) = zf(z)$ with a Stieltjes function~$f$.}

The shifted linear systems $(t I + A)\vx(t) = \vb$  can be solved in various ways, and here we focus on Krylov methods that are accelerated by a \emph{sketch-and-solve approach}; see, e.g., \rev{\cite{sarlos2006improved,balabanov2019randomized,balabanov2021randomized,balabanov2021randomizedblock,nakatsukasa2021fast,balabanov2022randomized}. 
The workhorse of sketching is an embedding matrix $S\in\mathbb{C}^{s\times N}$ \rev{with $s\ll N$} that distorts the Euclidean norm $\|\,\cdot\,\|$ of vectors in a controlled manner~\cite{sarlos2006improved,woodruff2014sketching}.} More precisely, given a positive integer $m$ and some $\varepsilon\in [0,1)$, we assume that $S$ is such that for all vectors~$\vv$ in the Krylov space $\mathcal{K}_{m+1}(A,\vb) := \Span\{\vb, A\vb, A^2\vb, \dots, A^{m}\vb\}$,
\begin{equation}
(1-\varepsilon) \| \vv \|^2 \leq \| S \vv\|^2 \leq (1+\varepsilon) \|\vv\|^2.
\label{eq:sketch}
\end{equation}
The matrix $S$ is also called an \emph{$\varepsilon$-subspace embedding} for $\mathcal{K}_{m+1}(A,\vb)$. 
Condition~\eqref{eq:sketch} can equivalently be stated with the Euclidean inner product: for all $\vu,\vv \in \mathcal{K}_{m+1}(A,\vb)$ 
\begin{equation}{\label{eq:sketch_innerproduct}}
\langle \vu, \vv \rangle - \varepsilon \| \vu\|\cdot \|\vv\|
            \leq \langle S\vu, S\vv \rangle 
            \leq \langle \vu, \vv \rangle + \varepsilon \| \vu\|\cdot \|\vv\|.
\end{equation}
 Of course, in practice, the matrix~$S$ is not explicitly available (not least as it requires knowledge of $\spK_{m+1}(A,\vb)$, which is only available in the final Krylov iteration $m$), and we hence have to draw it at random to achieve \eqref{eq:sketch} with high probability. 

In section~\ref{sec:sfom} below we focus our attention on the full orthogonalization method (FOM, \cite{Saad1992,Saad2003}) generalized to matrix functions $f(A)\vb$ using an integral representation of $f$. We show that the sketched FOM approximant admits a closed-form expression which is attractive for numerical evaluation and also allows us to bound the distance of this approximant to the full (non-sketched) FOM appoximant.   
In section~\ref{sec:sgmres} we use the generalized minimal residual method (GMRES, \cite{SaadSchultz1986}) to derive a sketched GMRES approximant that often exhibits a more stable convergence behavior than sketched FOM but requires numerical quadrature for its practical evaluation. We  prove convergence of these approximants for the important class of Stieltjes functions~$f$ and positive real matrices~$A$.  Section~\ref{sec:impl} is devoted to the discussion of implementation details. Following our previous work~\cite{FrommerGuettelSchweitzer2014a} we discuss how the sketched FOM and GMRES approximants can be evaluated using adaptive numerical quadrature. Section~\ref{sec:exp} contains discussions of some numerical experiments for medium and large-scale problems. We conclude in section~\ref{sec:concl} and provide an outlook on future work.

\section{Sketched FOM approximation}\label{sec:sfom}

The basis of polynomial Krylov methods for the approximation of $f(A)\vb$  is the   \emph{Arnoldi method}~\cite{Arnoldi1951}. Applying $m$ Arnoldi iterations with $A$ and $\vb$ yields the \emph{Arnoldi relation}
\begin{equation}\label{eq:arnoldi_relation}
AV_m = V_m H_m + h_{m+1,m}\vv_{m+1}\ve_m^T,
\end{equation}
with 
$
V_m = \left[\vv_1, \vv_2, \ldots, \vv_m \right] \in \CNm
$
containing an orthonormal basis of the Krylov space 
$
\spK_m(A,\vb) := \Span\{\vb, A\vb, A^2\vb, \dots, A^{m-1}\vb\},
$
and $[V_m,\vv_{m+1}]$ being an orthonormal basis of $\mathcal{K}_{m+1}$. The matrix $H_m$ is unreduced upper-Hessenberg and $\ve_m$ denotes the $m$th canonical unit vector in $\mathbb{R}^m$. 

The \emph{Arnoldi (or FOM) approximation} $\vf_m \approx f(A)\vb$ is obtained by projecting the original problem onto the Krylov space and evaluating $f$ for a small $m\times m$ matrix: 
\begin{equation}\label{eq:arnoldi_approximation}
\vf_m := V_mf(V_m^\dagger AV_m)V_m^\dagger\vb,
\tag{FOM}
\end{equation}
where  $V_m^\dagger$ is the Moore--Penrose inverse of $V_m$. Due to orthonormality of the basis~$\{\vv_j\}$, we have $V_m^\dagger = V_m^H$ and $\vf_m =\|\vb\|V_mf(H_m)\ve_1$. However, it will be useful to write \eqref{eq:arnoldi_approximation} in this more general form with a possibly nonorthonormal~$V_m$. 

The evaluation of~\eqref{eq:arnoldi_approximation} requires the storage of the full Krylov basis $V_m$, i.e., $m$~vectors of size~$N$. The (modified) Gram--Schmidt orthogonalization process to compute the orthonormal Krylov basis $V_m$ requires $O(Nm^2)$ arithmetic operations.  
For sufficiently large $N$, memory requirements and orthogonalization time  impose a limit on the maximal number $m_{\max}$ of Krylov iterations that can be performed, and thereby a limit on the attainable accuracy of the \rev{FOM}  approximation. In the Lanczos method~\cite{Lanczos1950} for Hermitian matrices~$A$, the cost of orthogonalization is just $O(Nm)$ due to the short-term recurrence of the Krylov basis, but if the full vector $f(A)\vb$ needs to be approximated, a memory requirement of $O(Nm)$ generally remains. (A notable exception is the case $f(z)=z^{-1}$ where the short recurrence for the Lanczos vectors translates into a short recurrence for the iterates, resulting in the famous conjugate gradient method~\cite{HestenesStiefel1952}.)

Using the integral representation of $f(H_m)= f(V_m^\dagger A V_m)$  in~\eqref{eq:arnoldi_approximation}, 
\[
    \vf_m = \int_\Gamma  \|\vb\| V_m (t I + H_m)^{-1} \ve_1 \d\mu(t) 
    =  \int_\Gamma  \vx_m(t) \d\mu(t), 
\]
we find that the integrand contains the FOM \rev{(or Galerkin)}  approximants 
\begin{equation}\label{eq:galerkin}
\vx_m(t) := \|\vb\| V_m (t I + H_m)^{-1} \ve_1 :=  V_m \vy_m(t)
\end{equation}
for the solution $\vx(t)$ of the shifted linear systems $(t I + A)\vx(t)=\vb$. The residuals of these approximants are explicitly given by 
\[
    \vr_m(t) = \vb - (tI + A) \vx_m(t) = -\| \vb\| h_{m+1,m}(e_m^T (t I + H_m)^{-1}\ve_1)\vv_{m+1},
\]
i.e., $\vr_m(t) = \alpha(t) \vv_{m+1}$ is orthogonal to $\mathrm{span}(V_m)$. Now, instead of imposing this orthogonality condition fully, we  propose to merely  require that the sketched residual $S\vr_m(t)$ be orthogonal to the sketched span of the Krylov basis, $\mathrm{span}(SV_m)$, where $S$ is an $s\times N$ sketching matrix. \rev{This is the same as the sketched Galerkin orthogonality condition for parametric linear systems used in \cite{balabanov2019randomized}. More precisely,} we require that
\[
 \widehat \vx_m(t) = V_m\widehat \vy_m(t) \quad \text{with} \quad (SV_m)^H [S\vb - S(t I + A) \widehat \vx_m(t)] = \vnull,
\]
or equivalently (if the inverted quantity is well defined),
\begin{equation}\label{eq:ym_hat}
    \widehat \vx_m(t) = V_m\widehat \vy_m(t) \quad \text{with} \quad  
    \widehat \vy_m(t) = [(SV_m)^H (tSV_m + SAV_m)]^{-1} (SV_m)^H (S\vb).
\end{equation}
The \emph{sketched FOM approximant to $f(A)\vb$} is then naturally defined as
\begin{equation}
    \widehat \vf_m :=  \int_\Gamma  \widehat \vx_m(t) \d\mu(t)
    = V_m \int_\Gamma   [(SV_m)^H (tSV_m + SAV_m)]^{-1}\d\mu(t)  (SV_m)^H (S\vb).
    \label{eq:sfom}\tag{sFOM}
\end{equation}

Some immediate comments are in order.
\smallskip
\begin{enumerate}
    \item If $S=I$, \eqref{eq:arnoldi_approximation} and \eqref{eq:sfom} are the same approximants. 
    \item The sketched orthogonality condition is imposed explicitly in \eqref{eq:sfom}, hence there is no requirement for the Krylov basis $V_m$ to be orthogonal. This means that  $V_m$ can be constructed without orthogonalization or by using a truncated orthogonalization procedure; see section~\ref{sec:fomcompl}.
    \item The sketched matrices $SV_m$ and $SAV_m$ can be constructed on-the-fly during the Arnoldi iteration, being expanded by $S \vv_{m+1}$ and $S A \vv_{m+1}$  when the  new Krylov basis vector $\vv_{m+1}$ is appended to $V_m$. The matrix-vector product $A \vv_{m+1}$ can be reused in the following iteration so that the overall number of matrix-vector products with $A$ remains the same as for the Arnoldi procedure without sketching.
    \item If the full vector approximation $\widehat\vf_m$ defined by \eqref{eq:sfom} is needed, then $V_m$ will still need to be stored as $\widehat \vx_m(t) = V_m \widehat \vy_m(t)$. However, as opposed to the standard FOM approach, $V_m$ does not need to be (fully) orthogonal and hence  $V_m$ can be held on slow memory (e.g., hard disk). Full access to $V_m$ is only needed once the sketched FOM approximant $\widehat\vf_m$  is formed, but not during the basis generation. Alternatively, the sketched approximation also makes it viable to use a two-pass approach~\cite{Borici2000,FrommerSimoncini2008} in the  case of non-Hermitian~$A$.
    \item  If only a few (say, $\ell \ll N$) selected components of $\widehat\vf_m$ are needed or, more generally, a matrix-vector product $M \widehat\vf_m$ with a short matrix $M\in\mathbb{C}^{\ell\times N}$, then with truncated Arnoldi only $k+1$ basis vectors $\vv_j$ need to be kept in memory in addition to the small matrix $M V_m$.
\end{enumerate}

\subsection{A closed formula for sketched FOM} 

We now investigate the expression defining \eqref{eq:sfom} in a bit more detail. Provided that  \eqref{eq:ym_hat} is well defined, it is guaranteed that $SV_m$ is of full rank $m$ and that  $V_m^H\SHS V_m$ is nonsingular. We can therefore rewrite the expression appearing in square brackets   in~\eqref{eq:ym_hat} as 
\begin{align*}
& [tV_m^H\SHS V_m + V_m^H\SHS AV_m]^{-1} \\
&= (V_m^H\SHS V_m)^{-1}[tI + V_m^H\SHS AV_m(V_m^TS^TSV_m)^{-1}]^{-1},
\end{align*}
so that~\eqref{eq:sfom} can be further rewritten as
\begin{align}
\widehat \vf_m &= V_m \int_\Gamma [tV_m^H\SHS V_m + V_m^H\SHS AV_m]^{-1}\ \d\mu(t)\ (SV_m)^H (S\vb) \nonumber\\
           &= V_m (V_m^H\SHS V_m)^{-1}\int_\Gamma 
           [tI + V_m^H\SHS AV_m(V_m^H\SHS V_m)^{-1}]^{-1}\ \d\mu(t)\ (SV_m)^H (S\vb) \nonumber\\
           &= V_m (V_m^H\SHS V_m)^{-1} f\left(V_m^H\SHS AV_m(V_m^H\SHS V_m)^{-1}\right) (SV_m)^H (S\vb).\tag{sFOM'}\label{eq:sfom1}
\end{align}
Note that~\eqref{eq:sfom1} is  a closed formula for the sketched approximation, not involving any integration, just like the standard FOM approximation~\eqref{eq:arnoldi_approximation}. 

Both \eqref{eq:sfom} and~\eqref{eq:sfom1} are completely independent of the choice of  $V_m$ as long as $\mathrm{span}(V_m) = \mathcal{K}_m(A,\vb)$. As $SV_m$ is  of full rank $m$, for our analysis we may require without loss of generality that the sketched basis be orthonormal, i.e., \begin{equation}
\label{eq:white}
(SV_m)^H SV_m=I_m.
\end{equation}
In this case we obtain a much simpler expression 
\begin{equation}
    \widehat \vf_m =  V_m  f\left(V_m^H\SHS AV_m\right) V_m^H \SHS  \vb.
    \label{eq:sfom2}\tag{sFOM''}
\end{equation}

The ``basis whitening'' condition \eqref{eq:white} was first recommended in \cite{rokhlin2008fast}, and it is also used in~\cite{balabanov2022randomized,nakatsukasa2021fast} to stabilize sketched GMRES and eigenvalue computations. In \cite{balabanov2022randomized}, the basis whitening condition is enforced during the Gram--Schmidt orthonormalization process on sketched vectors. But it can also be imposed retrospectively {\color{black} at a lower computational cost:} if $SV_m= Q_m R_m$ is a thin QR decomposition of the (nonorthonormal) sketched basis $S V_m$, we simply replace
\[
    SV_m \leftarrow Q_m, \quad SAV_m  \leftarrow (SAV_m) R_m^{-1},  \quad V_m  \leftarrow V_m R_m^{-1} \ \text{\rev{(only implicitly!)}}
\]
in \eqref{eq:sfom2}, resulting in 
\begin{equation}\label{eq:sfom3}
    \widehat \vf_m =  V_m  \big(R_m^{-1} f\left(Q_m^H  S AV_m R_m^{-1}\right) Q_m^H  S  \vb\big).
    \tag{sFOM'{'}'}
\end{equation}

\revv{We remark that if $SV_m$ and hence $R_m$ are extremely ill-conditioned, it might be safer to replace $R_m^{-1}$ with a numerical pseudoinverse (though we have not found a need for that in any of our numerical tests reported in section~\ref{sec:exp}).}

\rev{
\subsection{Algorithm and computational complexity}\label{sec:fomcompl}
\begin{algorithm}[t]
\caption{\label{alg:sketched_fom}Sketched FOM approximation of  $f(A)\vb$}
\begin{algorithmic}[1]
\setstretch{1.2}
\smallskip

\Statex \textbf{Input:} \ \ $A\in\mathbb{C}^{N\times N}$, $\vb\in\mathbb{C}^N$, function $f$, integers $m < s\ll N$ 
\Statex \textbf{Output:} $\widehat{\vf}_m \approx f(A)\vb$
\State Draw sketching matrix $S \in \C^{s \times N}$
\State Generate (nonorthogonal) basis $V_m$ of $\mathcal{K}_m(A,\vb)$, as well as $SV_m$ and $SAV_m$ \label{line:orthog}
\State Compute thin QR decomposition $SV_m= Q_m R_m$\label{line:thin_qr}
\State $\widehat \vf_m \leftarrow  V_m \big(R_m^{-1} f\left(Q_m^H  S AV_m R_m^{-1}\right) Q_m^H  S  \vb\big)$\label{line:f}
\end{algorithmic}
\end{algorithm}

{A concise summary of our sketched FOM algorithm is given in Algorithm~\ref{alg:sketched_fom}. 
One of the simplest ways to generate the nonorthogonal Krylov basis $V_m$ is to use a \emph{$k$-truncated modified Gram--Schmidt method} whereby at any iteration $j$, the vector $\vv_{j}$ is orthogonal to the previous $k$ basis vectors $\vv_{j-1},\vv_{j-2},\ldots,\vv_{j-k}$ only (with vectors having nonpositive indices ignored). 

We now  discuss the computational cost of Algorithm~\ref{alg:sketched_fom} assuming that
\begin{itemize}
    \item $A$ is a sparse matrix with $O(N)$ nonzeros, 
    \item $V_m$ is computed using  $k$-truncated  Gram--Schmidt  where $k = O(1)$, and
    \item the sketching parameter is chosen as $s = O(m)$.
\end{itemize} 
Under these assumptions, computing the basis $V_m$ in line~\ref{line:orthog} of Algorithm~\ref{alg:sketched_fom} has a computational cost of $O(Nmk^2) = O(Nm)$. The cost of sketching $V_m$ and $AV_m$ depends on the specific choice of the sketching matrix $S$. For example, the subsampled random Fourier transform~\cite{WoolfeLibertyRokhlinTygert2008}, see also~\cite[Sec.~2.3.1]{nakatsukasa2021fast}, can be applied using $O(Nm\log m)$ arithmetic operations. Performing the thin QR decomposition in line~\ref{line:thin_qr} has cost $O(sm^2) = O(m^3)$. 

Crucially, for the computation in line~\ref{line:f} of Algorithm~\ref{alg:sketched_fom},  the full basis $V_m$ should  \emph{not} be transformed to $V_m R_m^{-1}$ explicitly, as this would incur a rather high cost of $O(Nm^2)$, the same as standard Gram--Schmidt orthogonalization. Instead, it is possible to compute the compressed matrix operating only on matrices of size $s \times m$ and $m \times m$, resulting in a computational complexity of $O(m^3)$. Evaluating $f$ on a (dense) matrix of size $m \times m$ typically also has a cost of $O(m^3)$ and finally afterwards forming the approximation $\widehat{\vf}_m$ requires another $O(Nm + m^2) = O(Nm)$ operations. 

In total, Algorithm~\ref{alg:sketched_fom} computes the sketched FOM approximant with a cost of at most $O(Nm\log m + m^3)$.}

\begin{remark}
Truncated (or incomplete) Arnoldi methods were first considered by Saad in the context of eigenvalue problems~\cite[Sec.~3.2]{Saad1980b} and linear systems~\cite[Sec.~3.3]{Saad1981}. For computations involving matrix functions, truncated Arnoldi has mostly been used for approximating the action of the matrix exponential, $\exp(-tA)\vb$, in exponential integrators, were the effects of incomplete orthogonalization can be attenuated by choosing a smaller time step $t$; see, e.g.,~\cite{GaudreaultRainwaterTokman2018,Koskela2015}. 
In all our experiments in this paper, we use $k$-truncated Arnoldi, with $k$ ranging from $2$ to $4$. 

There also exist other possibilities for constructing Krylov bases using a limited number of inner products (or no inner products at all), e.g., based on recurrence relations for Chebyshev polynomials~\cite[Sec.~4]{JoubertCarey1992} or Newton polynomials~\cite[Sec.~4]{PhilippeReichel2012}. For comparisons of the different approaches, we refer the reader to~\cite[Sec.~4]{nakatsukasa2021fast} and in particular the numerical experiments in~\cite[Sec.~5]{PhilippeReichel2012}. 
\end{remark}

\begin{remark}\label{rem:alice}
An alternative approach for computing randomized FOM approximants has been proposed recently and independently in \cite{cortinovis2022speeding}. Starting with the Arnoldi relation~\eqref{eq:arnoldi_relation}, one immediately finds that
\[
    V_m^\dagger A V_m = H_m + h_{m+1,m} V_m^\dagger \vv_{m+1}\ve_m^T.
\]
Now, $\mathbf{x}_m = V_m^\dagger \vv_{m+1}$ can be computed by solving a least-squares problem $\mathbf{x}_m = \arg \min_{\vx \in \mathbb{C}^m} \| \vv_{m+1} - V_m \mathbf x \|$, or one can cheaply approximate it by sketching:
\[
\widehat{\mathbf{x}}_m = \arg \min_{\vx \in \mathbb{C}^m} \| S\vv_{m+1} - (SV_m) \mathbf x \|.
\]
The authors of \cite{cortinovis2022speeding} then suggest to use the approximation $$\widehat\vf_m = V_m f(H_m +  h_{m+1,m} \widehat{\mathbf{x}}_m \ve_m^T) \|\vb\| \ve_1,$$ which turns out to be mathematically equivalent to our  sketched FOM approximant; see the discussion in \cite[section~3.3]{cortinovis2022speeding}. 
\end{remark}

\smallskip

Interestingly, the alternation  interpretation  in \cite{cortinovis2022speeding} allows us to characterize $\widehat\vf_m$ as an approximation obtained from the \emph{Arnoldi-like decomposition} 
\begin{eqnarray}
    A V_m &=& V_m (H_m +  h_{m+1,m} \widehat{\mathbf{x}}_m \ve_m^T) + h_{m+1,m} (\vv_{m+1} - V_m \widehat{\mathbf{x}}_m) \ve_m^T \nonumber\\
    &=:& V_m \widehat{H}_m  + h_{m+1,m}\widehat{\vv}_{m+1} \ve_m^T\label{eq:H_hat}
\end{eqnarray}
defined in \cite[eq.~(2.5)]{EiermannErnst2006}. By \cite[Thm.~2.4]{EiermannErnst2006}, we have the following result.
\begin{corollary}\label{cor:alice}
The sketched FOM approximant $\widehat{\mathbf{f}}_m$ to $f(A)\vb$ satisfies 
\begin{equation}\label{eq:interp}
\widehat{\mathbf{f}}_m = V_m f(\widehat{H}_m) \|\vb\| \ve_1 = q(A)\vb,
\end{equation}
where $q$ is the unique polynomial of degree at most $m-1$ that interpolates $f$ at the eigenvalues of $\widehat{H}_m$ defined in~\eqref{eq:H_hat}. 
\end{corollary}

The interpolation characterization~\eqref{eq:interp} applies to \emph{any} approximation obtained from an Arnoldi-like decomposition $A V_m = \widehat{H}_m  + h_{m+1,m}\widehat{\vv}_{m+1} \ve_m^T$, including the ``cheaper approximants'' suggested in \cite[section~3.2]{cortinovis2022speeding}. This can even be generalized to so-called \emph{Krylov-like decompositions} which drop the requirement that the columns of $V_m$ be linearly independent~\cite{EiermannErnstGuettel2011}. This interpolation characterization is numerically robust, independent of the basis conditioning: for example, it has been used in \cite{AfanasjewEtAl2008b} to analyze the convergence of restarted/truncated  Krylov approximants when the truncation length is small as~$k=1$. Perhaps these insights can be used in future work to analyze sketched Krylov approximations in the case where the basis $V_m$ becomes numerically singular. 

}

\subsection{Error analysis} 
It is interesting to compare \eqref{eq:sfom2}   to~\eqref{eq:arnoldi_approximation}. Clearly,   both formulas  coincide if 
$V_m^H \SHS  = V_m^\dagger$, but we can also state a more general result.

\smallskip

\begin{corollary}\label{cor:err}
Assume that \eqref{eq:sketch} holds with $\varepsilon\in [0,1)$ and that $S V_m$ has orthonormal columns, i.e., \eqref{eq:white} is satisfied. Let $\vf_m$ and $\widehat \vf_m$ denote the FOM and sketched FOM approximants to $f(A)\vb$ defined in \eqref{eq:arnoldi_approximation} and \eqref{eq:sfom2}, respectively. Then
\[
 \|\vf_m - \widehat \vf_m\| \leq \sqrt{\frac{{1+\varepsilon}}{{1-\varepsilon}}}\cdot  \|\vb\| \cdot \| f(V_m^\dagger A V_m) -  f\left(V_m^H\SHS AV_m\right) \|.
\]
\end{corollary}
\begin{proof}
First note that $S\vb$ is a column in the sketched Krylov matrix $S V_m$, the latter of which is assumed to be orthonormal. Hence, $\| (S V_m)^H S\vb \| = \|S\vb\| \leq \sqrt{1+\varepsilon} \|\vb\|$ by  \eqref{eq:sketch}. Also by \eqref{eq:sketch} we have 
\[
    \|V_m\|  = \max_{\|\vw\|=1} \|V_m \vw\| \leq  \max_{\|\vw\|=1} \frac{1}{\sqrt{1-\varepsilon}} \|SV_m \vw\| = \frac{1}{\sqrt{1-\varepsilon}} \| SV_m \| = \frac{1}{\sqrt{1-\varepsilon}}.
\]
The claimed inequality follows from 
\[
 \|\vf_m - \widehat \vf_m\|  \leq \| V_m \| \cdot \| f(V_m^\dagger A V_m) -  f\left(V_m^H\SHS AV_m\right) \| \cdot \| (S V_m)^H S\vb \|.
\]
\hfill\end{proof}

\smallskip

\begin{remark}\label{rem:fom}
We stress again that the sketched FOM approximants \eqref{eq:sfom} and \eqref{eq:sfom1} with an arbitrary Krylov basis $V_m$ will yield exactly the same errors $ \|\vf_m - \widehat \vf_m\|$  as the approximants \eqref{eq:sfom2} assuming orthonormal $SV_m$, but only in the later case we obtain a simple error formula as in Corollary~\ref{cor:err}.   

The corollary offers a general avenue for a thorough analysis of the distance between the full FOM and the sketched FOM approximation depending on the sketching matrix $S$ and the function $f$. However, without some restrictive assumptions on $S$ and $f$, the factor $\| f(V_m^\dagger A V_m) -  f\left(V_m^H\SHS AV_m\right) \|$ will likely be difficult to bound: while it is clear that the Rayleigh quotient $V_m^\dagger A V_m$ has eigenvalues contained in the numerical range $W(A):=\{ \vx^H A \vx : \|\vx\|=1\}$, the eigenvalues of $V_m^H\SHS AV_m$, \rev{which by Corollary~\ref{cor:alice} are the nodes of an interpolating polynomial for $f$,} are not restricted to such a canonical set. In light of~\eqref{eq:sketch_innerproduct}, the only inclusion that we can give without further assumptions is
\begin{equation}\label{eq:inclusion_sketched_ritz}
\Lambda(V_m^H\SHS AV_m) \subset W(A) + \Delta(0,\varepsilon \|A\|) = \{z_1 + z_2 : z_1 \in W(A), |z_2| \leq \varepsilon \|A\| \}.
\end{equation}

Hence, even if $A$ is Hermitian, there is no guarantee that  $\Lambda(V_m^H\SHS AV_m)$ be real; or if $W(A)$ is contained in the right complex half-plane, then $V_m^H\SHS AV_m$ may still have eigenvalues with negative real part. This may lead to potential instabilities when evaluating the sketched FOM approximant. Indeed, we observe in numerical experiments reported in section~\ref{sec:exp} that sketched FOM can exhibit non-smooth convergence behavior on some problems. \rev{Similar observations have been reported for FOM (and  analyzed for symmetric positive definite~$A$) in the recent technical report \cite{timsit2023randomized}.}
\end{remark}

\section{Sketched GMRES approximation}\label{sec:sgmres}
\rev{Sketched GMRES methods for the solution of (parameterized) linear systems have been considered, e.g., in~\cite{balabanov2021randomized,balabanov2021randomizedblock,nakatsukasa2021fast,balabanov2022randomized}. In our setting with shifted linear systems}  $(tI + A) \vx(t)=\vb$, we simply impose that the  residual  \begin{equation}
    \widetilde \vr_m(t) := \vb - (t I + A) \widetilde \vx_m(t) \label{eq:res}
\end{equation}    
of  $\widetilde \vx_m(t) := V_m \widetilde \vy_m(t)$ be minimal after sketching:
\[
\|S\widetilde \vr_m(t) \| = \|    S \vb - S (tI + A) V_m \widetilde \vy_m(t) \|\to \min .
\]
The solution is
\[
\widetilde \vy_m(t) = (tSV_m + SAV_m)^\dagger (S\vb),
\]
leading to the \emph{sketched GMRES approximant to $f(A)\vb$} defined as 
\begin{align}
\widetilde \vf_m &=  \int_\Gamma \widetilde \vx_m(t) \d\mu(t)
=V_m \int_\Gamma (tSV_m + SAV_m)^\dagger \d\mu(t)  (S\vb).
\tag{sGMRES}\label{eq:sgmres}
\end{align}

We will see in numerical experiments reported in section~\ref{sec:exp} that the sketched GMRES method can exhibit a smoother convergence behavior than sketched FOM. On the other hand, there appears to be no simple closed form for the sketched GMRES approximant and quadrature is necessary for its evaluation.

\smallskip

\begin{remark}\label{rem:sgmres}
The approximant~\eqref{eq:sgmres} does \emph{not} necessarily coincide with the harmonic Arnoldi approximant introduced in \cite[Section~6]{FrommerGuettelSchweitzer2014a} even when $S=I$. In the harmonic Arnoldi approach the shifted linear system with $t=0$ is solved by GMRES, but all other systems with $t\in\Gamma$ are solved such that their residual vectors are collinear to that of the $t=0$ problem. There is no reason why the residuals $\widetilde \vr_m(t)$ defined in \eqref{eq:res} would necessarily be collinear for different values of~$t$. \revv{As a consequence, it also appears to be more challenging to interpret sGMRES as a simple polynomial interpolation process as we have done for sFOM in Corollary~\ref{cor:alice}.}
\end{remark}

\smallskip

Let us compare the  residual $\widetilde \vr_m(t)$ of the sketched solution $\widetilde \vx_m(t)$ to the residual $\vr_m(t)$ of the full GMRES solution 
$\vx_m(t) = V_m (tV_m + AV_m)^\dagger \vb$. We have
\[
        \|  \vr_m(t) \| \leq \|\widetilde \vr_m(t)\| \leq \frac{1}{\sqrt{1-\varepsilon}}  \|S \widetilde \vr_m(t) \| \leq  \frac{1}{\sqrt{1-\varepsilon}}  \|S \vr_m(t) \|\leq \sqrt{\frac{1+\varepsilon}{1-\varepsilon}} \| \vr_m(t) \|.
\]
For the second and fourth inequalities we have used \eqref{eq:sketch} and this is indeed valid because $\widetilde \vr_m(t), \vr_m(t)\in\mathcal{K}_{m+1}(A,\vb)$. In the first and third inequality we have used the fact that $\vr_m(t)$ and $S\widetilde \vr_m(t)$ have smallest possible norm as per definition, respectively. 
Crucially,
\begin{equation}
\|\widetilde \vr_m(t)\| \leq C_\varepsilon \| \vr_m(t) \|, \quad C_\varepsilon := \sqrt{\frac{1+\varepsilon}{1-\varepsilon}} \quad \text{for all \ $t\in\Gamma$}.
\label{eq:resbnd}
\end{equation}

\subsection{Convergence for Stieltjes functions of positive real matrices}
In this section, let us assume that $A$ is a positive real matrix, i.e.,  $\real(\vv^H A \vv) > 0$ for all $\vv \in \mathbb{C}^N,\ \vv\neq 0$. Further, assume that $f$ is a Stieltjes function with $\Gamma=[0,+\infty)$. 
Building on the analysis in \cite{FrommerGuettelSchweitzer2014b}, the  quantities
\begin{align*}
   \delta &:= \lambda_{\min}\left(\frac{A+ A^{H}}{2}\right) = \min \big\{ \real\big(\vv^H A\vv\big): \|\vv\|=1\big\}, \\
   \rho &:= \lambda_{\min}\left(\frac{A^{-1}+ A^{-H}}{2}\right) = \min \big\{ \real\big(\vv^H A^{-1}\vv\big): \|\vv\|=1\big\},
\end{align*}
will be useful. Since with $A$ the matrices $A^{-1}$ and $\AHA$ are also positive real, the numbers  $\delta$ \rev{and} $\rho$ are positive. 

For a Hermitian matrix $M$, let us define the $M$-energy norm as $\|\vv\|_M := \sqrt{ \vv^H M \vv}$, and denote by $\widetilde \ve_m(t):=(tI+A)^{-1}\vb - \widetilde \vx_m(t)$ and $\ve_m(t):=(tI+A)^{-1}\vb - \vx_m(t)$ the error of the sketched GMRES and full GMRES approximants, respectively. Then we can equivalently write the residual inequality~\eqref{eq:resbnd} in terms of errors as 
\begin{equation}
    \|\widetilde \ve_m(t)\|_{(A+tI)^H(A+t I)} \leq C_\varepsilon   \| \ve_m(t)\|_{(A+tI)^H(A+t I)}.
\label{eq:resbnd_err}
\end{equation}
The following lemma from \cite[Lemma~6.4]{FrommerGuettelSchweitzer2014b} is included for convenience. 

\smallskip 

\begin{lemma} \label{lem:AHA_norm_estimates}
Let $A\in\mathbb{C}^{N\times N}$ be positive real. 
\begin{itemize}
  \item[(i)] For all $\vv \in \CN$ and $t \geq 0$ we have 
 \[ 
\| \vv \|_{\AHA}^2 \leq  \frac{1}{\rev{\|A\|^{-2}} t^2 + 2\rho t +1} \| \vv \|_{(A+tI)^H(A+tI)}^2\, .
\]
\item[(ii)] For $t \geq 0$ we have
\[
\frac{1}{\rev{\|A\|^{-2}} t^2 + 2  \rho t +1} \leq \frac{\rev{\|A\|}}{(t+\rho\rev{\|A\|^2})^2}\,.
\]
\end{itemize}
\end{lemma}

We are now in the position to state our main theorem on the convergence of the sketched GMRES approximation. The proof will be different to that for the harmonic Arnoldi approximation presented in \cite[Theorem~6.5]{FrommerGuettelSchweitzer2014b} as we do not have collinearity of residuals for different values of $t\geq 0$; cf.~Remark~\ref{rem:sgmres}. 

\smallskip

\begin{theorem}\label{thm:conv}
Let $A$ be a positive real matrix and $f$ a Stieltjes function. Assume that the condition \eqref{eq:sketch} holds with $\varepsilon\in [0,1)$.  Let $\widetilde\vf_m$ be the sketched GMRES approximant to $f(A)\vb$ defined by \eqref{eq:sgmres}. Let $\beta_0 = \arccos(\delta/\|A\|)\in [0,\pi/2)$. Then
\[
\|f(A)\vb - \widetilde \vf_m \|_{\AHA} 
\leq 
C_1 C_\epsilon \|\vb\| (\sin(\beta_0))^m,
\]
with constants $C_1 = \rev{\|A\|}f(\rho \rev{\|A\|^2})$ and $C_\varepsilon = \sqrt{(1+\varepsilon)/(1-\varepsilon)}$.
\end{theorem}

\smallskip

\begin{proof}
We have
\[
\|f(A)\vb - \widetilde \vf_m \|_{\AHA} 
=
\left\| \int_0^{\infty} \widetilde \ve_m(t) \d \mu(t) \right\|_{\AHA} 
\leq 
\int_0^{\infty}  \frac{ C_\varepsilon \| \ve_m(t) \|_{(A+tI)^H(A+tI)} }{\sqrt{\rev{\|A\|^{-2}} t^2 + 2  \rho t +1}}   \d \mu(t) 
\]
\[
 = \int_0^{\infty} \frac{ C_\varepsilon \| \vr_m(t) \| }{\sqrt{\rev{\|A\|^{-2}} t^2 + 2  \rho t +1}}   \d \mu(t),
\]
where we have used Lemma~\ref{lem:AHA_norm_estimates}(i) together with~\eqref{eq:resbnd_err} for the first inequality. 
It remains to bound $\|\vr_m(t)\|$, the residual of the standard GMRES method for the system $(tI+A)\vx(t) =\vb$. 
Using a convergence result in \cite{Elman1982} (see also \cite{beckermann2005some} for an improved version), we have the bound 
\[
    \|\vr_m(t)\| \leq  \|\vb\| (\sin(\beta_t))^m, 
\]
\[
    \cos(\beta_t) = \frac{\lambda_{\min}([(tI+A)+(tI+A)^H]/2)}{\|tI + A\|}
    =
    \frac{t + \delta}{\|tI + A\|}
    <1.
\]
Since for all $t\geq 0$, 
\[
\cos(\beta_t) 
\geq 
\frac{t + \delta}{t + \|A\|}
\geq
\frac{\delta}{\|A\|}
= \cos(\beta_0),
\]
we also have $\beta_t \leq \beta_0$ and hence
\[
 \|\vr_m(t)\| \leq  \|\vb\| (\sin(\beta_0))^m \quad \text{for all} \ t\geq 0.
\]
Therefore,
\[
\|f(A)\vb - \widetilde \vf_m \|_{\AHA} 
\leq
 \int_0^{\infty} \frac{  C_\varepsilon  \|\vb\| (\sin(\beta_0))^m }{\sqrt{\rev{\|A\|^{-2}} t^2 + 2  \rho t +1}}   \d \mu(t)  
 \leq
 C_1 C_\varepsilon  \|\vb\| (\sin(\beta_0))^m,
\]
where $C_1= \rev{\|A\|} f(\rho \rev{\|A\|^2})$  by Lemma~\ref{lem:AHA_norm_estimates}(ii).
\hfill\end{proof}

\smallskip

While the convergence factor $\sin(\beta_0)$ in Theorem~\ref{thm:conv} can often be improved, results like this can generally not be  expected to give particularly tight error bounds. This is not a problem of our derivation but  common to all a-priori convergence bounds on GMRES. Nevertheless, Theorem~\ref{thm:conv} guarantees convergence of the sketched GMRES approximant \eqref{eq:sgmres} for Stieltjes functions of positive real matrices.

One might wonder why we have used $\delta = \lambda_{\min}((A+A^H)/2)$ in place of the distance of the origin to the numerical range,  $\mathrm{dist}(0,W(A))$, as sharper convergence factors could be obtained with the later; see \cite{beckermann2005some}. This is because the former expression increases exactly by $t$ if $A$ is replaced by $t I+A$, while the latter only satisfies
\[
\mathrm{dist}(0,W(tI + A)) \leq t + \mathrm{dist}(0,W(A)),
\]
an inequality in the wrong direction to be of use in the proof of  Theorem~\ref{thm:conv}.

\section{Implementation details}\label{sec:impl}
In this section we discuss a number of topics concerning the implementation of the sketched FOM and sketched GMRES methods. To support this discussion, we summarize the quadrature-based  sketched GMRES method in Algorithm~\ref{alg:sketched_gmres} (including the truncated modified Gram--Schmidt process).

\begin{algorithm}[t]
\caption{\label{alg:sketched_gmres}Sketched GMRES approximation of  $f(A)\vb$ with $k$-truncated Arnoldi}
\begin{algorithmic}[1]
\setstretch{1.2}
\smallskip

\Statex \textbf{Input:}\ \  $A\in\mathbb{C}^{N\times N}$, $\vb\in\mathbb{C}^N$, function $f$, integers $m, s, \ell_1, \ell_2$, tolerance $\texttt{tol}$
\Statex \textbf{Output:} $\widetilde{\vf}_m \approx f(A)\vb$
\State Draw sketching matrix $S \in \C^{s \times N}$
\State $\vv_1 \leftarrow (1/\|\vb\|_2) \cdot \vb$
\State $\vw \leftarrow A\vv_1$
\State Compute sketches $S\vv_1$ and $S\vw$ \Comment{start construction of $SV_m$, $SAV_m$}

\smallskip

\For{$j = 1,\dots,m$}
	\For{$i = \max\{1,j-k+1\},\dots,j$} \Comment{truncated MGS orthogonalization}
        \State $\vw \leftarrow \vw - \langle\vv_i,\vw\rangle\vv_i$
	\EndFor
	\State $\vv_{j+1} \leftarrow (1/\|\vw\|_2) \cdot \vw$
	\State $\vw \leftarrow A\vv_{j+1}$
    \State Compute sketches $S\vv_{j+1}$ and $S\vw$ and append them to $SV_{j+1}$ and $SAV_{j+1}$
\EndFor

\smallskip

\State Compute thin QR decomposition $SV_m= Q_m R_m$ \Comment{basis whitening}
\State  $SV_m \leftarrow Q_m, \quad SAV_m  \leftarrow (SAV_m) R_m^{-1}, \quad V_m  \leftarrow V_m R_m^{-1} \ \text{\rev{(only implicitly!)}}$ 

\smallskip

\If{contour $\Gamma$ is not fixed} \Comment{can be skipped for Stieltjes functions}
\State Compute solutions $\Lambda$ of generalized rectangular EVP $SAV_m\vx = -\lambda SV_m\vx$
\State Choose $\Gamma$ such that it encircles $\Lambda$
\EndIf

\smallskip

\State Compute quadrature rules $\vq_{\ell_1}(S, A, V_m, \vb)$ and $\vq_{\ell_2}(S, A, V_m, \vb)$ \Comment{see~\eqref{eq:quadrature_rule}}

\While{$\| \vq_{\ell_1}(S, A, V_m, \vb) - \vq_{\ell_2}(S, A, V_m, \vb) \| > \texttt{tol}$} 
\State Set $\vq_{\ell_1}(S, A, V_m, \vb) \leftarrow \vq_{\ell_2}(S, A, V_m, \vb)$ \Comment{reuse previous result}
\State Set $\ell_1 \leftarrow \ell_2$, $\ell_2 \leftarrow \lfloor\sqrt{2}\cdot\ell_2\rfloor$ \Comment{increase order of quadrature rules}
\State Compute quadrature rule $\vq_{\ell_2}(S, A, V_m, \vb)$
\EndWhile

\smallskip

\State $\widetilde{\vf}_m \leftarrow V_m\vq_{\ell_2}(S, A, V_m, \vb)$
\end{algorithmic}
\end{algorithm}

\subsection{Adaptive quadrature}\label{subsec:quad}

In order to evaluate the sketched GMRES approximant~\eqref{eq:sgmres}, the occurring integral needs to be approximated as no closed form is available (in contrast to the situation for the sketched FOM approximant). One can in principle use any $\ell$-point quadrature rule
\begin{equation}\label{eq:quadrature_rule}
\int_\Gamma (tSV_m + SAV_m)^\dagger(S\vb) \d\mu(t) \approx \sum_{i=1}^\ell w_i (t_i SV_m + SAV_m)^\dagger(S\vb) =: \vq_\ell(S, A, V_m, \vb),
\end{equation}
with weights $w_i$ and quadrature nodes $t_i \in \Gamma$ ($i=1,2,\dots,\ell$). Following the implementation of the  quadrature-based  restarted Arnoldi method in~\cite{FrommerGuettelSchweitzer2014a}, we propose to use a rather simple form of numerical quadrature: we start by computing the result of two quadrature rules $\vq_{\ell_1}(S, A, V_m, \vb)$ and $\vq_{\ell_2}(S, A, V_m, \vb)$ of orders $\ell_1 < \ell_2$, respectively. If
\begin{equation}\label{eq:quadrature_tol}
\| \vq_{\ell_1}(S, A, V_m, \vb) - \vq_{\ell_2}(S, A, V_m, \vb) \| < \texttt{tol}
\end{equation}
for some user-specified tolerance $\texttt{tol}$, we accept the result of the higher-order quadrature rule $\vq_{\ell_2}$ and use it to approximate $\widetilde{\vf}_m \approx V_m \vq_{\ell_2}(S, A, V_m, \vb)$. Should~\eqref{eq:quadrature_tol} not be satisfied, we increase the order of both quadrature rules by setting $\ell_1 \leftarrow \ell_2$ and $\ell_2 \leftarrow \lfloor\sqrt{2}\cdot\ell_2\rfloor$. This way, the result of the previous computation can be reused for the updated $\vq_{\ell_1}(S, A, V_m, \vb)$, while only $ \vq_{\ell_2}(S, A, V_m, \vb)$ needs to be computed anew. This process is repeated until~\eqref{eq:quadrature_tol} is fulfilled. We emphasize that all computations related to the adaptive quadrature rule are done on small matrices of size $s \times m$, while quantities of size~$N$ are only formed once the quadrature is  sufficiently accurate.

A suitable choice of a specific quadrature rule should depend on $f$ and $\Gamma$. We refer the reader to~\cite[Section~4]{FrommerGuettelSchweitzer2014a} for a discussion of quadrature rules tailored specifically to the important functions $\exp(A)$, $A^{-\alpha}$, and $A^{-1}\log(I+A)$, the latter two of which are Stieltjes functions. When $f$ is not a Stieltjes function, one additionally needs to construct a suitable contour $\Gamma$ before numerically integrating~\eqref{eq:sgmres}. 

\subsection{Two-pass $k$-truncated Arnoldi computation}\label{subsec:two_pass}

For Hermitian $A$, the two-pass Lanczos method~\cite{Borici2000,FrommerSimoncini2008} is a simple approach for employing the Lanczos method in a limited memory environment by running the iteration twice. The sketched FOM and sketched GMRES methods allow us to use a similar approach in the nonsymmetric case: we employ an Arnoldi method with truncation length~$k$ for computing the Krylov basis $V_m$. Whenever we compute a new basis vector $\vv_j$, we compute the sketches $S\vv_j$ and $SA\vv_j$, thereby assembling the matrices $SV_m$ and $SAV_m$ column-by-column. As soon as a basis vector is not needed any longer for performing the truncated orthogonalization, we discard it from memory. At the end of this first pass of the method, we approximate the coefficient vector
\[
\widetilde{\vy}_m = \int_\Gamma (tSV_m + SAV_m)^\dagger \d\mu(t)  (S\vb)
\]
by adaptive quadrature as outlined in section~\ref{subsec:quad}. In the second pass, the sketched GMRES approximation is computed as
$$
\widetilde{\vf}_m = \sum\limits_{i=1}^m [\widetilde{\vy}_m]_i \vv_i,
$$
which can be updated from one iteration to the next. Thus, we can again discard old basis vectors. Of course, this approach doubles the number of matrix-vector products that need to be performed, but it often converges in much fewer iterations than a restarted Arnoldi approach, which amortizes the additional work.

\subsection{Stopping criterion}\label{subsec:stopping_criterion}
As in any iterative method, it is important to have an estimate for the approximation error available in order to determine whether the computed quantity $\widetilde{\vf}_m$ is an accurate enough approximation of the desired quantity $f(A)\vb$ (or to be able to stop the iteration early, if fewer iterations than initially expected are required for reaching the desired accuracy). A-priori error bounds as given in Theorem~\ref{thm:conv} are not well-suited for this purpose, as they tend to overestimate the actual error norm by a large margin and also involve quantities that are usually difficult to access; see also the brief discussion at the end of section~\ref{sec:sgmres}.

A simple error estimate that is often used in Krylov  methods is the difference of two iterates, i.e.,
\[
\big\|f(A)\vb-\widetilde{\vf}_m\big\| \approx \big\|\widetilde{\vf}_{m+d}-\widetilde{\vf}_m\big\|
\]
for a small integer $d \geq 1$. In the context of sketched GMRES, it is important to be able to evaluate a stopping criterion without access to the full matrix $V_m$ (e.g., when it is kept in slow memory or when a two-pass approach is employed). Thus, it must be avoided to explicitly form $\widetilde{\vf}_{m+d}$ and $\widetilde{\vf}_{m}$. To do so, \rev{we exploit that $\widetilde{\vf}_m, \widetilde{\vf}_{m+d} \in \spK_{m+d}(A,\vb)$ and $S$ is an $\varepsilon$-subspace embedding for this space. Therefore, by~\eqref{eq:sketch},
\begin{equation}\label{eq:sketched_err_est1}
\revv{\frac{1}{\sqrt{1+\varepsilon}} \big\|S(\widetilde{\vf}_{m+d}-\widetilde{\vf}_m)\big\|}  \leq 
\big\|\widetilde{\vf}_{m+d}-\widetilde{\vf}_m\big\| \leq \frac{1}{\sqrt{1-\varepsilon}}\big\|S(\widetilde{\vf}_{m+d}-\widetilde{\vf}_m)\big\|
\end{equation}
Writing $\widetilde{\vf}_{m} = V_m\widetilde{\vy}_m$, we obtain from~\eqref{eq:sketched_err_est1} the relation
\begin{equation}\label{eq:sketched_err_est2}
\big\|\widetilde{\vf}_{m+d}-\widetilde{\vf}_m\big\| \leq \frac{1}{\sqrt{1-\varepsilon}}\big\|SV_{m+d}\widetilde{\vy}_{m+d}-SV_m\widetilde{\vy}_m\big\| = \left\| SV_{m+d}\left(\widetilde{\vy}_{m+d}-\left[\begin{array}{c} \widetilde{\vy}_{m} \\ \vnull_d\end{array}\right]\right)\right\|,
\end{equation}
which can be evaluated without access to the full basis, working just with small-scale vectors and matrices. For estimating the unknown embedding quality $\varepsilon$ one can, e.g., compare $\|S\vv_j\|, j = 1,\dots,m$ to $\|\vv_j\| = 1$ whenever a new basis vector is computed and keep track of these values.}

\section{Numerical tests}\label{sec:exp}

In this section we demonstrate the stability and
efficiency of the proposed sketching approaches on some model problems and problems from relevant applications. All computations were
performed in MATLAB R2022A. Timings are measured on a PC with an AMD Ryzen 7 3700X 8-core CPU with clock rate 3.60GHz and 32 GB RAM. Since a part of MATLAB code is interpreted, MATLAB implementations are not always best suited for comparing running times of algorithms, but they are certainly appropriate to assess stability. Moreover, since all algorithms spend most of their time in sparse matrix-vector multiplications, which are calls to precompiled libraries, larger
differences in running times can be trusted as significant. 
All Krylov bases are generated by a (truncated) modified Gram--Schmidt process without reorthogonalization. In all examples with sketching, we use the basis whitening condition~\eqref{eq:white}. The code used for generating all figures and tables in this section is available at \url{https://github.com/marcelschweitzer/sketched_fAb}.

\subsection{Convection--diffusion example}\label{subsec:convdiff}

In this example, let $A$ be the discretization of a two-dimensional convection-diffusion operator on the unit square with constant convection field pointing in the direction $[1,-1]$ and diffusion coefficient $D = 10^{-3}$, where we discretize the convection term by a first-order upwind scheme, giving
\begin{equation*}
A = \frac{D}{h^2} \cdot \rev{(I \otimes L + L \otimes I)} + \frac{1}{h} \cdot (C\oplus C^T) \in \mathbb{R}^{n^2\times n^2}
\end{equation*}
with $h=1/(n+1)$, $L  = \text{tridiag}(-1, 2, -1) \in \mathbb{R}^{n\times n}$, and $C = \text{tridiag}(-1, 1, 0) \in \mathbb{R}^{n\times n}$. For this experiment, we use $n = 100$. We approximate $A^{-1/2}\vb$, where $\vb$ is a vector of all ones scaled to have norm~$1$. For the sketching matrix   we use a subsampled randomized discrete cosine transform (DCT): $S = PFE$, where  
$E\in\mathbb{R}^{N\times N}$ is a diagonal matrix having diagonal entries $\pm 1$ with equal probability, $F\in\mathbb{R}^{N\times N}$ is a DCT, and $P\in\mathbb{R}^{s\times N}$ selects $s$ rows of $FE$ at random; see also \cite[Sec~8.1.1.]{nakatsukasa2021fast}.  The sketching parameter is fixed at $s = 2m_{\max}$, where $m_{\max} = 200$ is the maximum Krylov dimension we encounter.

Figure~\ref{fig:convdiff} illustrates the results. We compare the sketched FOM and GMRES approximations  to the best approximation obtained by explicitly projecting $f(A)\vb$ onto the Krylov space~$\mathcal{K}_m(A,\vb)$. In the sketched FOM case we test both the integral representation \eqref{eq:sfom} evaluated via quadrature as well as the closed formula~\eqref{eq:sfom2}. For sketched GMRES only an integral representation is available and we again use quadrature for its evaluation.

\begin{figure}
    \centering
    \begin{minipage}{0.49\textwidth}
    \hspace*{-3mm}\includegraphics[width=1.1\textwidth]{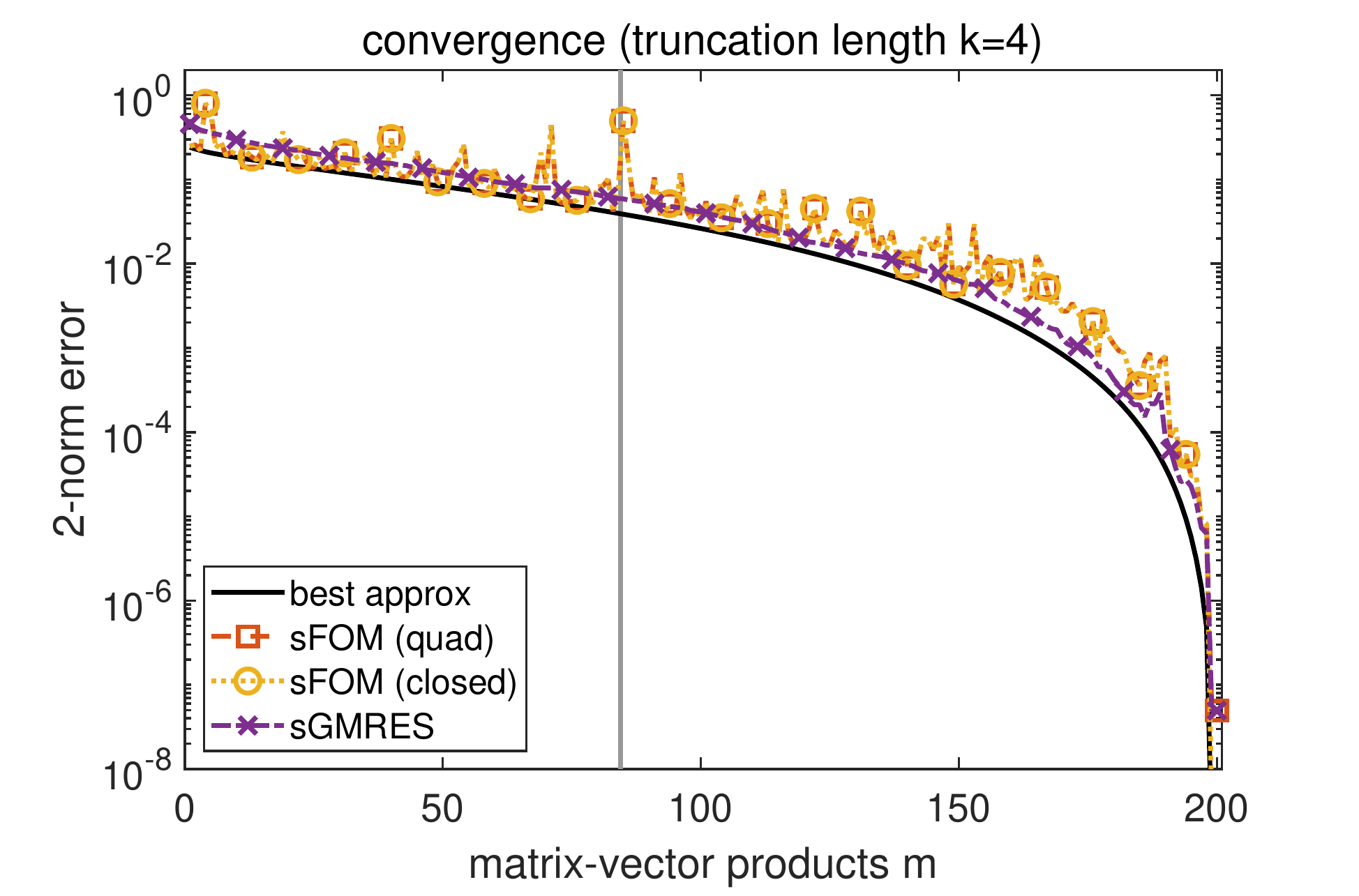}
    \end{minipage}
    \begin{minipage}{0.49\textwidth}
    \hspace*{-0mm}\includegraphics[width=1.1\textwidth]{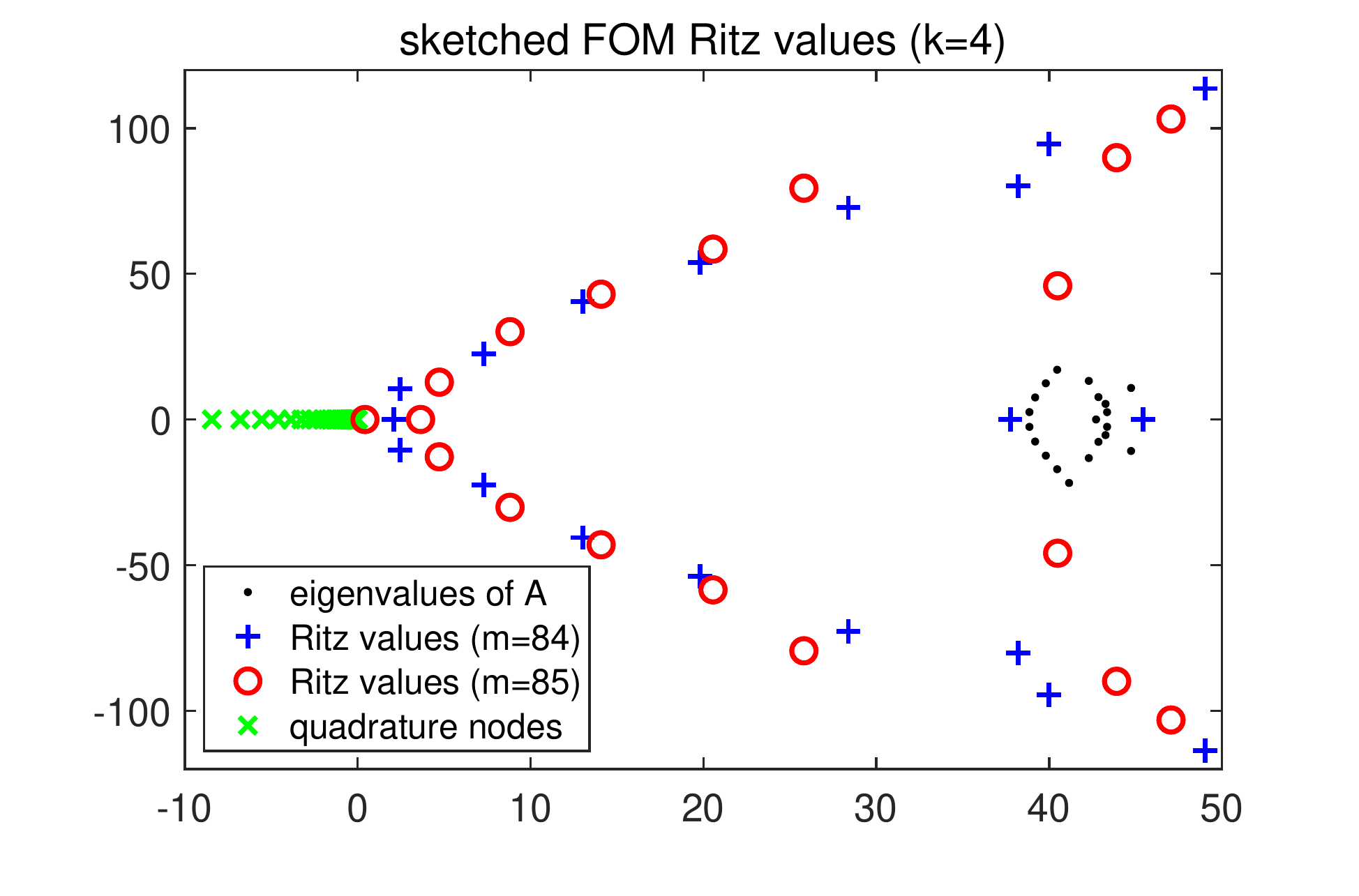}
    \end{minipage}
    \caption{Convection--diffusion example. The left plot shows the convergence of the sketched methods based on truncated Arnoldi with truncation parameter $k=4$. The error of best approximation to $f(A)\vb$ from the Krylov space $\mathcal{K}_m(A,\vb)$ is also shown. On the right we show some of the Ritz values $\Lambda(V_m^H S^H SAV_m)$ for the orders $m=84$ and $m=85$, which are closest to the gray vertical bar at position $m=84.5$ on the left. The jump in the sketched FOM error at $m=85$ is caused by a Ritz value being very close to a quadrature node.}
    \label{fig:convdiff}
\end{figure}

In the quadrature-based methods, we use a Gauss--Chebyshev rule after applying the variable transformation $x = (1-t)/(1+t)$ which maps the interval $[0,+\infty)$ to $(-1,1]$; see also~\cite[Section~4.1]{FrommerGuettelSchweitzer2014a}. The number of quadrature nodes is determined adaptively as described in section~\ref{subsec:quad}. For simplicity, we have  determined the quadrature rule once for the maximum Krylov dimension $m_{\max} = 200$ and then kept it fixed for all iterations. This way, $\ell = 45$ quadrature nodes were used for all $m=1,\ldots,m_{\max}$.

As can be seen in the left plot of Figure~\ref{fig:convdiff}, the error of all sketched approximations follows that of the best approximation quite well and also inherits the superlinear convergence. The sketched FOM approximants show a \rev{less regular convergence} behavior than the sketched GMRES approximants, the latter following the best approximation error very closely. The error curves of the two sketched FOM approximants (quadrature-based and closed form) are visually indistinguishable, \rev{indicating that the quadrature rule is highly accurate and that the quadrature approximation can be ruled out as the source of the irregular sFOM convergence.}

To gain some insight into the less regular convergence behavior of both sketched FOM variants, we plot on the right of Figure~\ref{fig:convdiff} the Ritz values  $\Lambda(V_m^H S^H SAV_m)$ for the orders $m=84$ and $m=85$. Order $m=85$ is characterized by a spike in the error curve and we see that one of the corresponding Ritz values is very close, namely at $x\approx 0.392$, to a quadrature node at $t=-3.05\cdot 10^{-4}$. The Ritz values of order $m=84$, on the other hand, stay safely away from the quadrature nodes. Some of the eigenvalues $\Lambda(A)$ are also shown for information.

\subsection{Network example}\label{subsec:network}
We consider the nonsymmetric binary adjacency matrix \texttt{wiki-Vote} of size $N=8,\! 297$ in the SNAP collection \cite{snapnets}.  The function to compute is $f(A)\vb$ where $f(z)=e^{-z}$ and $\vb$ is the vector of all ones. As in the previous example, $S$ is a subsampled randomized DCT with sketching parameter fixed at $s=2 m_{\max} =100$, independent of $m$. For sFOM and sGMRES we run truncated Arnoldi with truncation parameter $k=2,3,4$. The resulting three convergence plots are shown in Figure~\ref{fig:network}.  For the construction of the quadrature rule for the integral representations we use the  approach from \cite{FrommerGuettelSchweitzer2014a}, with a parabolic integration contour $\Gamma$ parameterized as 
\[
\gamma(t) = a + it  - c\zeta^2, \quad t\in\mathbb{R},
\]
 and with the parameters $a,c$ chosen so that the Ritz values are surrounded. A fixed quadrature rule with $\ell=100$ nodes is used in all cases and some of its nodes are shown in the fourth plot of Figure~\ref{fig:network}. Note how some of the eigenvalues of $A$, in particular the outliers, are well approximated by some of the Ritz values.
 
 We also include the quadrature-based restarted FOM \rev{code}  \texttt{funm\_quad} \cite{FrommerGuettelSchweitzer2014a} with restart lengths $r=2,3,4$ in Figure~\ref{fig:network}. The overall memory requirement of the orthogonalization for truncated Arnoldi and \texttt{funm\_quad} are comparable when $k=r$, namely they both require the storage of $k+1=r+1$ Krylov basis vectors of size~$N$.  We find that even with a truncation length as low as $k=2$, all sketched methods exhibit a surprisingly robust convergence, while restarted FOM requires a restart length of at least $r=4$ to converge steadily. In all cases, the sketched methods follow quite closely the error of the best approximant obtained by projecting the exact $f(A)\vb$ onto $\mathcal{K}_m(A,\vb)$, while restarting prevents or delays the convergence. \rev{We also depict the condition number of the non-orthogonal Krylov basis $V_m$ (multiplied by the unit round-off $u \approx 2.2 \cdot 10^{-16}$). Interestingly, the sketched Krylov methods continue to perform well and converge without problems even when the condition number of the basis reaches (or even exceeds) $u^{-1}$, while this is typically mentioned as a source of instablities in the literature; see, e.g.,~\cite{nakatsukasa2021fast}.}

\begin{figure}
    \centering
    \begin{minipage}{0.49\textwidth}
    \hspace*{-3mm}\includegraphics[width=1.1\textwidth]{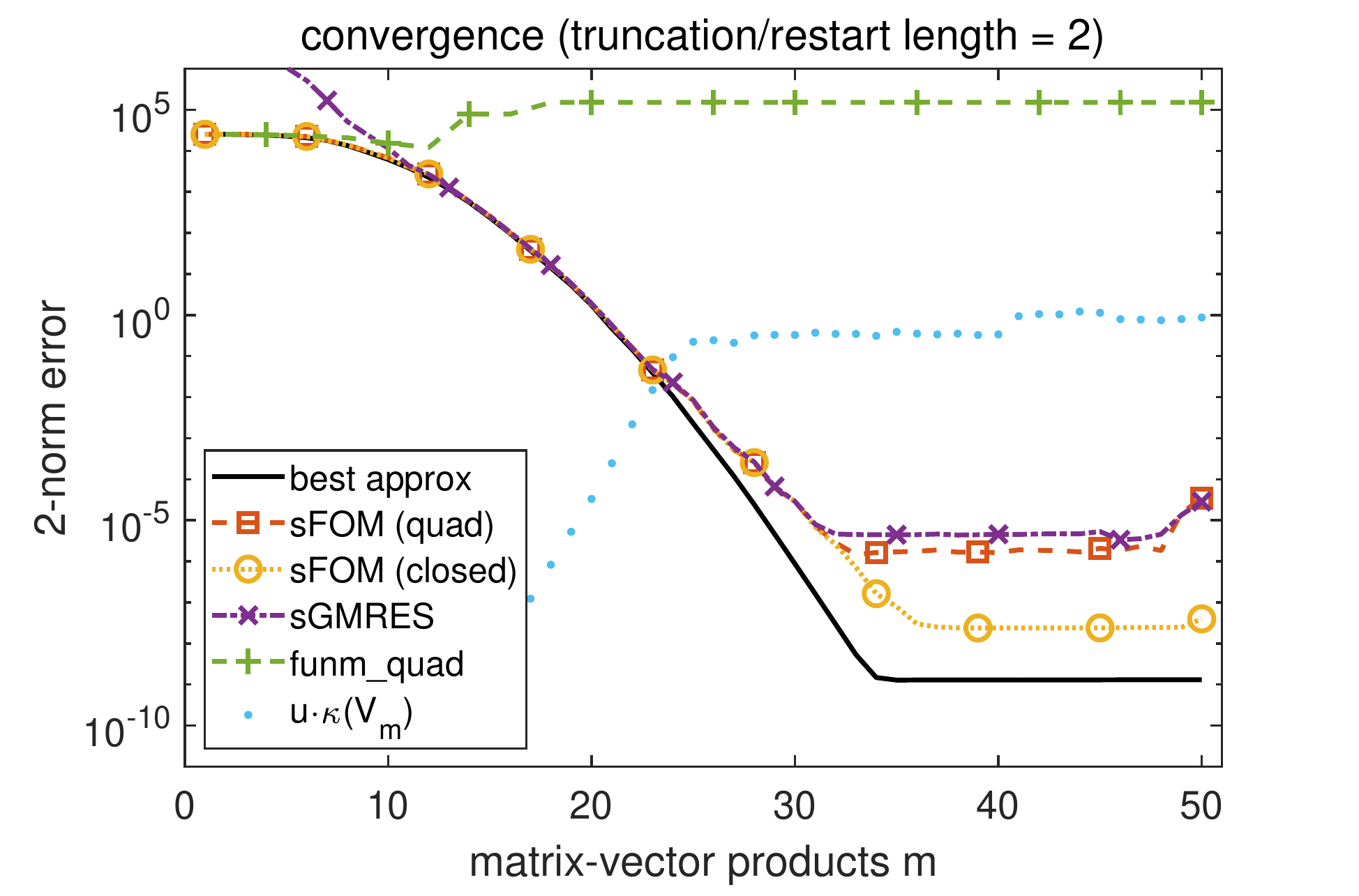}
    \end{minipage}
    \begin{minipage}{0.49\textwidth}
    \hspace*{-0mm}\includegraphics[width=1.1\textwidth]{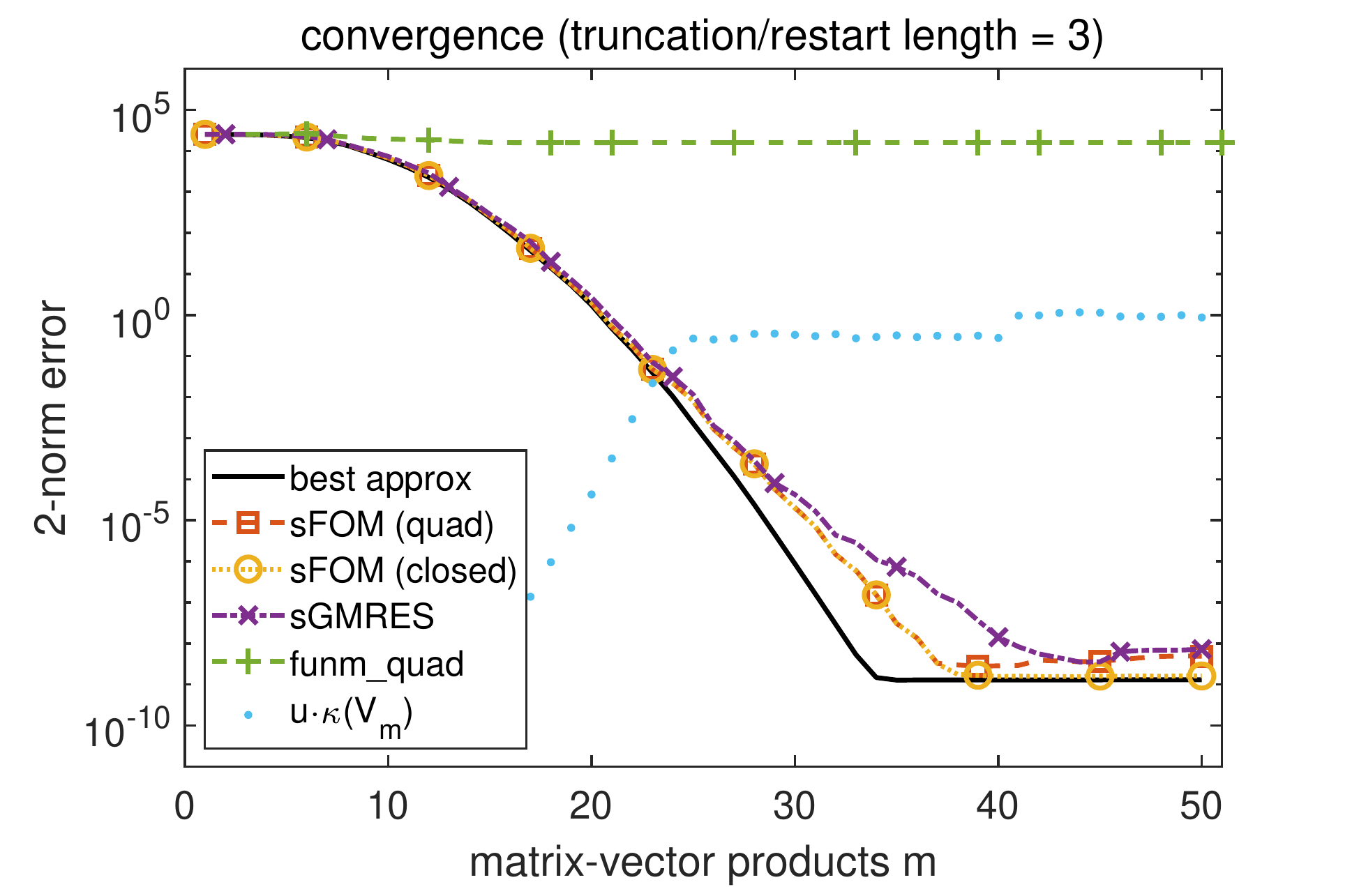}
    \end{minipage}
    \begin{minipage}{0.49\textwidth}
    \hspace*{-3mm}\includegraphics[width=1.1\textwidth]{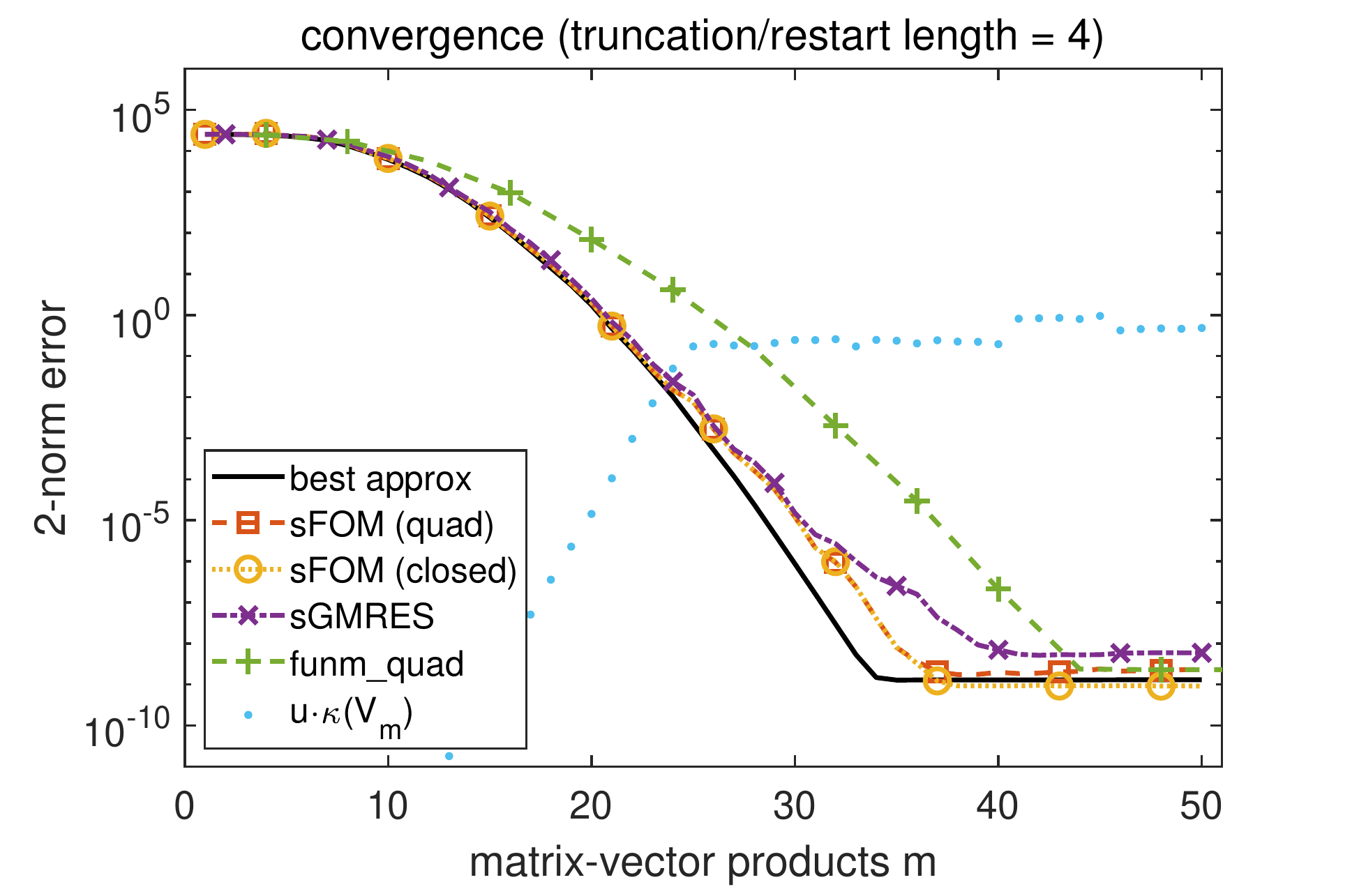}
    \end{minipage}
    \begin{minipage}{0.49\textwidth}
    \hspace*{-0mm}\includegraphics[width=1.1\textwidth]{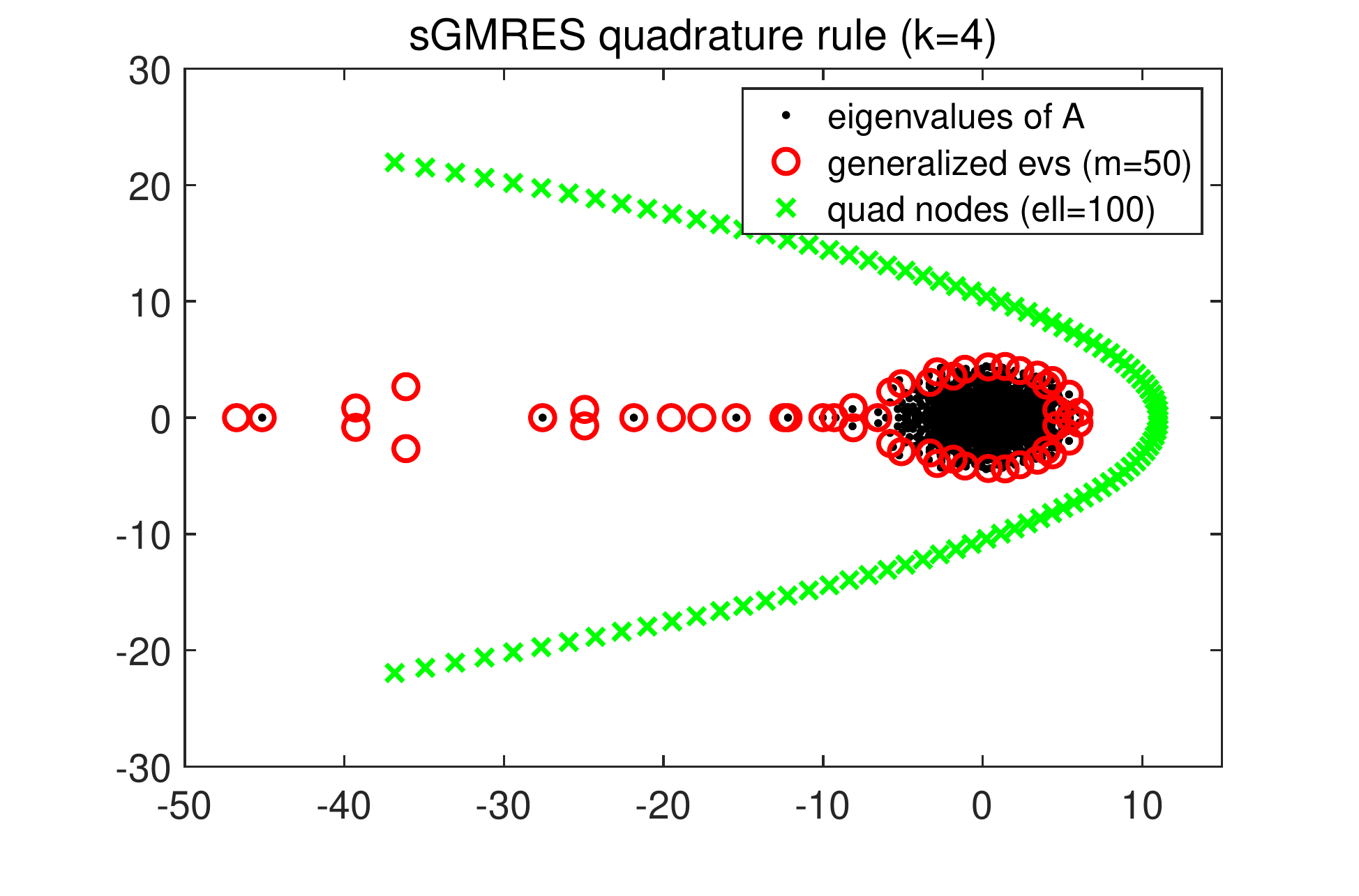}
    \end{minipage}
    \caption{Network example. The first three plots show the convergence of the sketched methods based on truncated Arnoldi with truncation parameter $k=2,3,4$. The error of the restarted Arnoldi approximation with restart length $r=k$ and the error of best approximation from the Krylov space $\mathcal{K}_m(A,\vb)$ is also shown\rev{, as well as the condition number of the truncated Krylov basis (multiplied by the unit round-off)}. The final plot shows the placement of $\ell=100$ complex-valued quadrature nodes on the parabolic contour relative to the Ritz values of order $m=50$.}
    \label{fig:network}
\end{figure}

\subsection{Lattice QCD}\label{subsec:qcd}

\begin{figure}
    \centering
    \begin{minipage}{0.49\textwidth}
    \hspace*{-5mm}\includegraphics[width=1.1\textwidth]{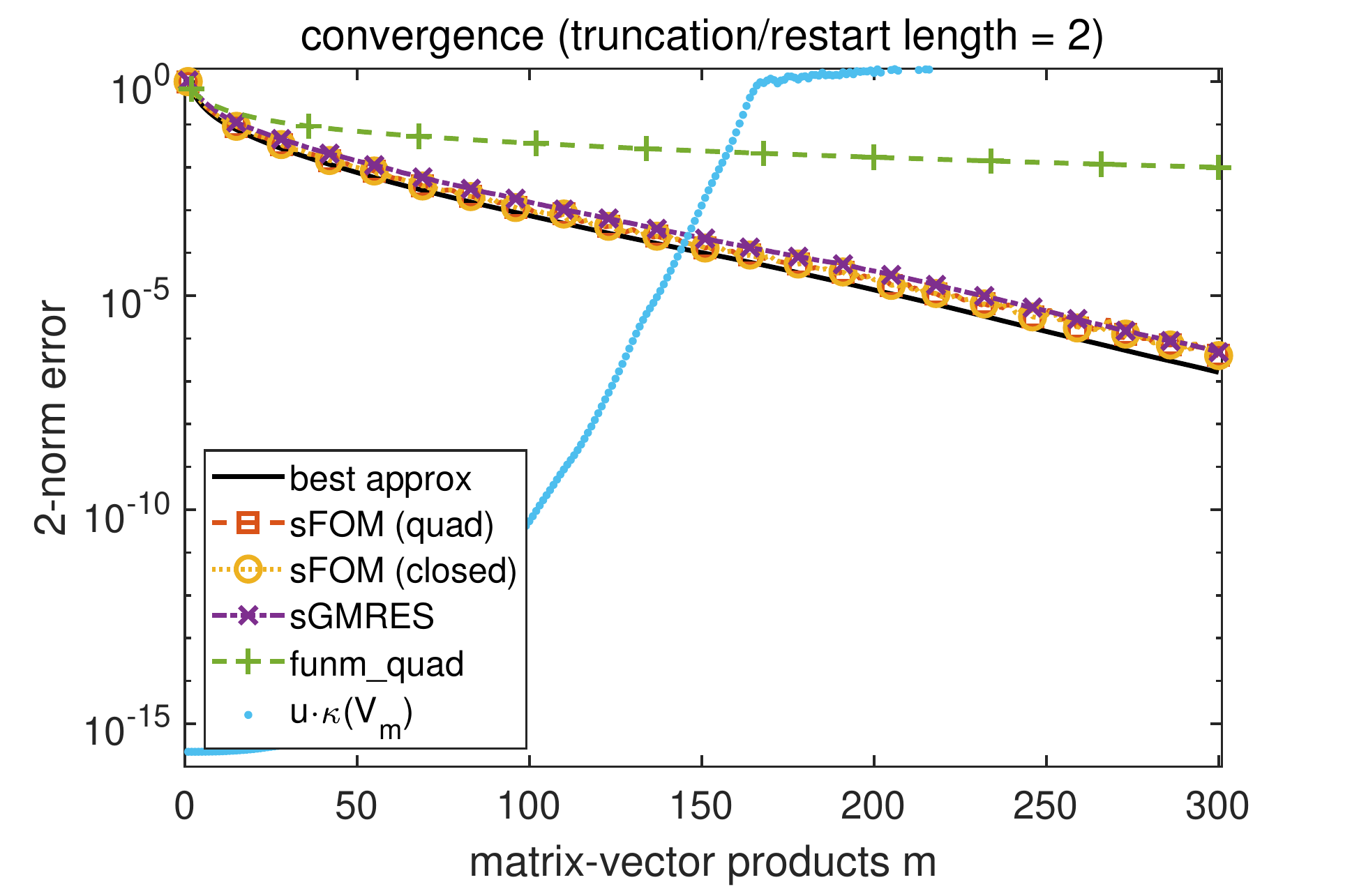}
    \end{minipage}
    \begin{minipage}{0.49\textwidth}
    \hspace*{-5mm}\includegraphics[width=1.1\textwidth]{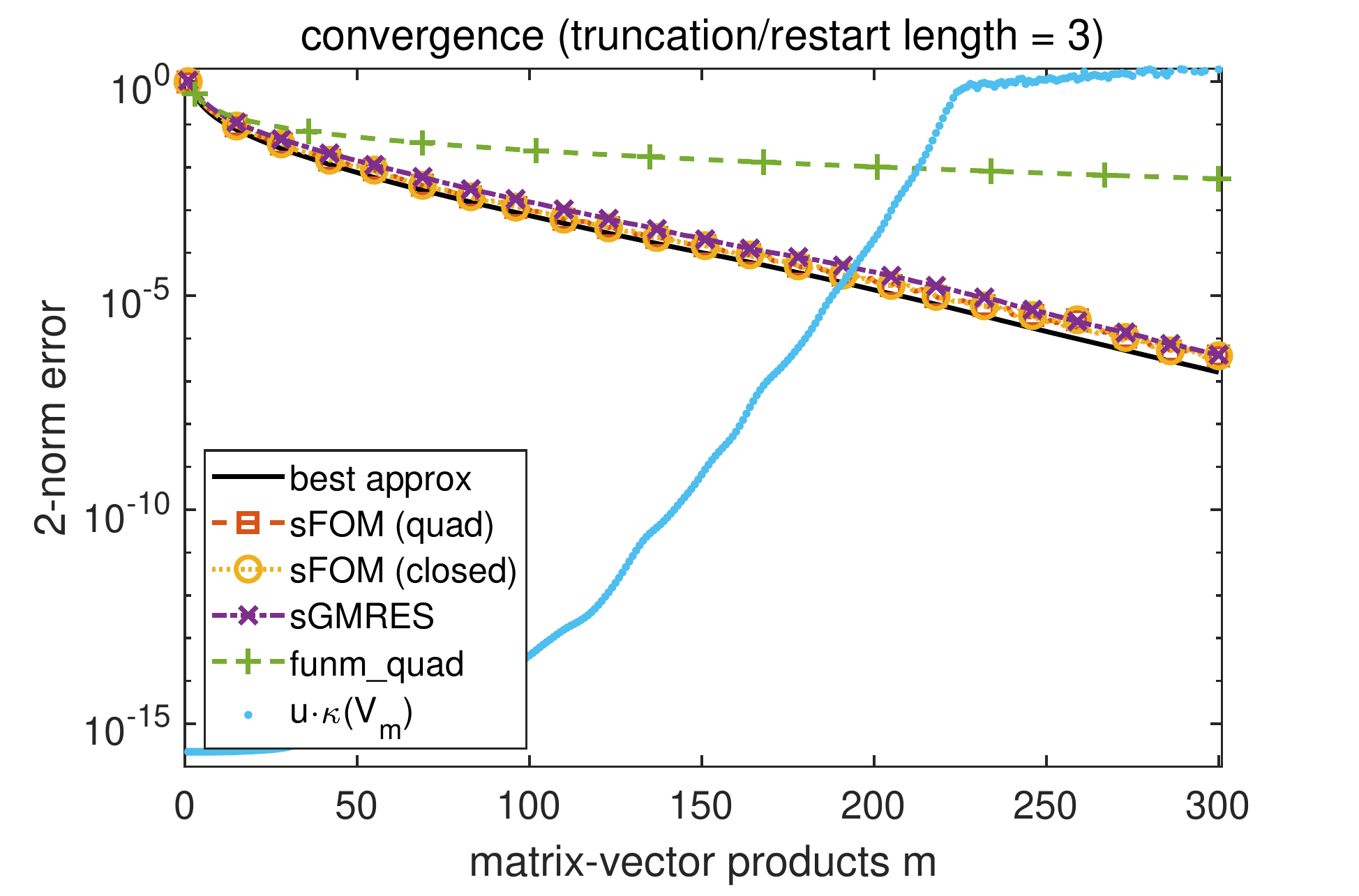}
    \end{minipage}
    \begin{minipage}{0.49\textwidth}
    \hspace*{-5mm}\includegraphics[width=1.1\textwidth]{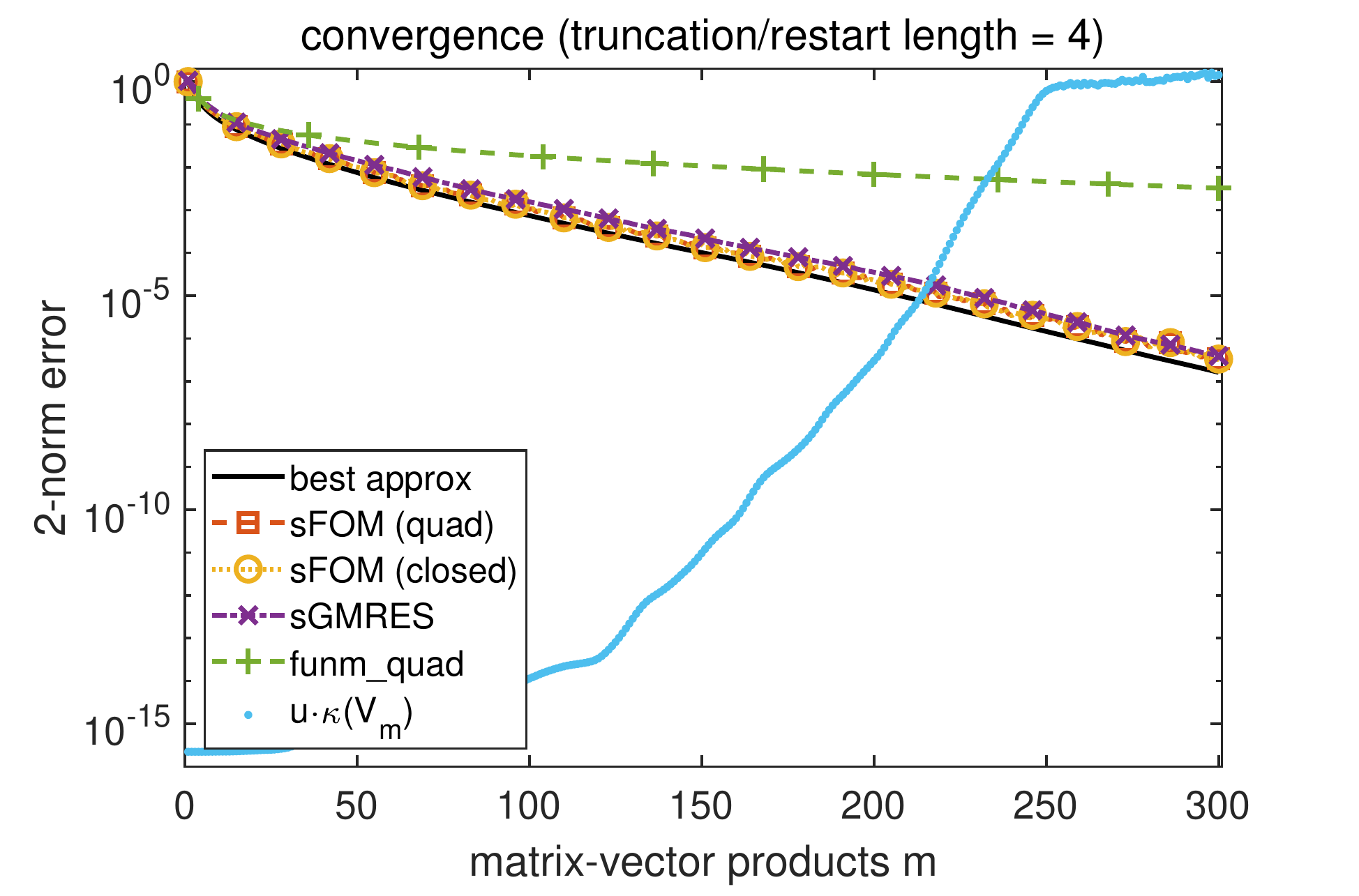}
    \end{minipage}
    \begin{minipage}{0.49\textwidth}
    \hspace*{-5mm}\includegraphics[width=1.1\textwidth]{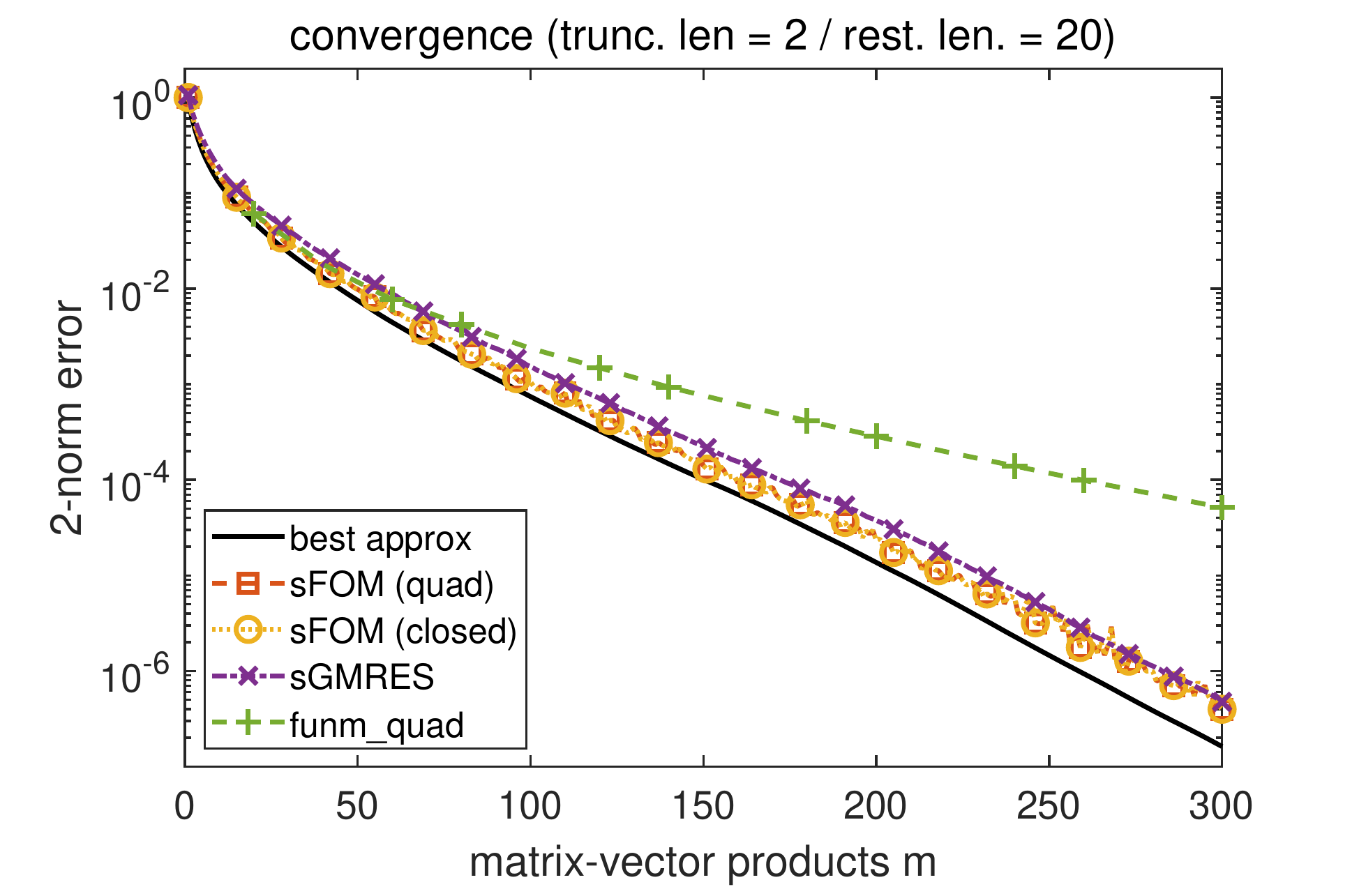}
    \end{minipage}
    \caption{Lattice QCD example. The first three plots show the convergence of the sketched methods based on truncated Arnoldi with truncation parameter $k=2,3,4$. The error of the restarted Arnoldi approximation with restart length $r=k$ and the error of best approximation from the Krylov space is also shown\rev{, as well as the condition number of the truncated Krylov basis (multiplied by the unit round-off)}. The final plot shows the convergence of the $4$-truncated sketched methods compared to the restarted Arnoldi method with $r = 20$.}
    \label{fig:qcd}
\end{figure}

\emph{Quantum chromodynamics} (QCD) is the area of theoretical physics that studies the strong interaction between quarks and gluons, governed by the Dirac equation. To be able to perform simulations, in \emph{lattice quantum chromodynamics}, the Dirac equation is discretized on a four-dimensional space--time lattice with 12 variables at each lattice point, corresponding to all possible combinations of three colors and four spins. In order to preserve the so-called \emph{chiral symmetry} on the lattice, one needs to solve linear systems with the overlap Dirac operator \cite{Neuberger1998},
\begin{equation}\label{eq:overlap_operator}
N_{\rm ovl} := \rho I + \Gamma_5\sign(Q).
\end{equation}
In~\eqref{eq:overlap_operator}, $\rho > 1$ is a mass parameter, $Q$ represents a periodic nearest-neighbor coupling on the lattice, and $\Gamma_5$ is a permutation matrix. The matrix $Q$ is very large, sparse, complex and, in the presence of a nonzero \emph{chemical potential} (the situation we consider here), non-Hermitian. 

As $\sign(Q)$ cannot be explicitly computed for realistic grid sizes, one typically solves linear systems with \eqref{eq:overlap_operator} by an inner-outer Krylov method which only needs to access $\sign(Q)$ via matrix-vector products. At each outer Krylov iteration, one therefore has to compute $\sign(Q)\vb$ where the vector $\vb$ changes from one iteration to the next. Efficient preconditioners for the ``outer iteration'' for~\eqref{eq:overlap_operator} can be constructed based on, e.g., domain decomposition and adaptive algebraic multigrid. It then turns out that the ``inner iteration'' for evaluating $\sign(Q)\vb$ represents the by far most expensive part of the overall computation (see, e.g.,~\cite[Section~5.2]{BrannickFrommerKahlLederRottmannStrebel2016}), which makes any improvements in this area very welcome.

To show how the sign function fits into the framework considered here, write
\begin{equation}\label{eq:sign_relation}
\sign(Q)\vb = (Q^2)^{-1/2}Q\vb.
\end{equation}
Thus, when performing a Krylov iteration with $A=Q^2$ (which of course does not need to be formed explicitly), we can use the same Gauss--Chebyshev rule for the inverse square root as in section~\ref{subsec:convdiff}. 
We use a lattice configuration with $8$ lattice points in the temporal and each spatial direction, resulting in $N = 12\cdot8^4 = 49,\!152$ and choose $\vb = \ve_1$ as the first canonical unit vector. 

For the first part of this experiment, we construct a fixed Gauss--Chebyshev quadrature rule with accuracy parameter $\texttt{tol} = 10^{-7}$, which results in $\ell = 176$ quadrature points.  We use a maximum Krylov dimension of $m_{\max}=300$ and a fixed sketching parameter $s = 2m_{\max} = 600$.
As before, we compare to the quadrature-based restarted Arnoldi method from~\cite{FrommerGuettelSchweitzer2014a}, which is also used in state-of-the-art HPC code for simulation of overlap fermions~\cite{BrannickFrommerKahlLederRottmannStrebel2016}. We use  the truncation parameters $k = 2, 3, 4$ and the same restart lengths $r=k$. 
Additionally, we compare the sketched methods with $k=2$-truncated Arnoldi to restarted FOM with restart length $r=20$, a value used in realistic large-scale simulations of overlap fermions.

The results of the experiment are depicted in the four plots of Figure~\ref{fig:qcd}. We observe that all sketched approximations converge robustly and follow the error of the best approximation closely, while convergence is strongly delayed in the restarted methods for $r = 2, 3, 4$ (although in contrast to the network  example, the restarted method does make progress for all restart lengths). Even for the larger restart length $r = 20$ (which leads to much higher orthogonalization cost than in the sketched methods), convergence is much slower than for the sketching-based  approaches. \rev{Additionally, we again observe that convergence of the sketched methods takes place well after the point where the basis condition number exceeds $u^{-1}$.}

\begin{table}
\centering
\caption{\rev{Lattice QCD example. Truncation (resp.\ restart) length, required number of Krylov iterations, wall-clock time and relative error norm at the final iteration for the different discussed algorithms when invoked with a target accuracy of $10^{-5}$. Details on the experimental setup are given in the final paragraphs of section~\ref{subsec:qcd}}.}
\label{tab:qcd}
\rev{\begin{tabular}{l|cccc}
method & $k$ resp.\ $r$ & Krylov dim.\ & time & rel.\ error \\[1mm]
\hline\hline\\[-2mm]
sketched FOM (closed form) & 2 & 220 & 3.36s & $5.93 \cdot 10^{-6}$ \\
sketched FOM (quadrature) & 2 & 240 & 7.66s & $3.17 \cdot 10^{-6}$ \\
sketched GMRES (quadrature) & 2 & 220 & 6.64s & $7.88 \cdot 10^{-6}$ \\
\texttt{funm\_quad} &  2 & 340 & 3.92s & $1.25 \cdot 10^{-5}$ \\
standard FOM & -- & 220 & 7.87s & $3.67 \cdot 10^{-6}$ \\
\end{tabular}}
\end{table}

In the second part of this experiment we measure the run time of the different methods, but now with all quadratures performed fully adaptively as explained in section~\ref{subsec:quad}. We use the same problem setup as before and aim for reaching an overall relative error norm below $10^{-5}$. We compare the run time of the sketched methods with truncation length $k = 2$ (i.e., at most 3 basis vectors need to be stored at a time) with that of restarted Arnoldi with restart length $r = 20$ (i.e., at most 21 vectors basis need to be stored at a time) \rev{and with standard FOM  which constructs an orthonormal basis of the Krylov space via modified Gram--Schmidt orthogonalization (but without reorthogonalization).} For a fair comparison, we check the error estimate \eqref{eq:sketched_err_est2} in the sketched methods \rev{(and a similar estimate in standard FOM)} every 20 iterations (i.e., $d = 20$), as \texttt{funm\_quad} also checks for convergence at the end of each restart cycle. I.e., in all quadrature-based methods (sketched or non-sketched), integrals need to be evaluated every 20th matrix-vector product. We stop the iteration once the error estimate is below the desired tolerance of $10^{-5}$. Note that the stopping condition in \texttt{funm\_quad} is also based on comparing approximants from subsequent restart cycles. As tolerance \texttt{tol} for the quadrature rules in the sketched methods we choose the same value $10^{-5}$ as for the desired error accuracy. In \texttt{funm\_quad}, we had to use the slightly more stringent tolerance of $10^{-6}$, as the method otherwise stagnated around an error norm of $10^{-4}$.

The results of this experiment are reported in Table~\ref{tab:qcd}. Among all methods, sketched FOM using the closed form~\eqref{eq:sfom2} runs the fastest, which is to be expected as it needs the smallest number of matrix-vector products, uses a short recurrence orthogonalization and has close to no overhead for things like quadrature. In comparison to restarted Arnoldi (\texttt{funm\_quad}), the second fastest method, sketched FOM saves about \rev{15\%} of run time and also reaches a higher accuracy. The quadrature-based sketching methods need slightly less than twice the time of restarted Arnoldi but also have much lower memory consumption. The bottleneck in the quadrature-based methods is the efficient evaluation of integrals, which make up the largest part of the run time. In the sketched GMRES approximation, approximately 42\% of the run time is spent in matrix-vector products, 52.5\% for computations related to evaluating the quadrature rule, 3.5\% for computing sketches, and 1.5\% for orthogonalization. \rev{Standard FOM (without sketching or restarting) is the slowest of all tested methods as---already for the rather moderate Krylov dimension in this example---orthogonalization cost becomes the dominating factor. For larger and more difficult problems where a higher Krylov dimension is necessary, the gains provided by sketching can be expected to be even more pronounced. This example clearly illustrates that sketching and restarting techniques are very relevant for nonsymmetric matrices even when memory is not limited.}

We end with a few further comments on how to best interpret the results above. In the QCD model problem we consider here, matrix-vector products are extremely expensive compared to inner products ($Q$ contains 49 nonzeros per row and needs to be applied twice per iteration). In situations were matrix-vector products are cheaper, the difference in run time between sketched FOM and restarted Arnoldi would be much higher, as orthogonalization then makes up a larger fraction of the cost in the restarted methods. The overhead in the quadrature-based methods can likely be reduced by a more sophisticated implementation. In particular, this would also be highly dependent on the computing environment (as quadrature rules can of course be evaluated in a  parallelized fashion) and is therefore beyond the scope of this work. Also keep in mind that the cost of quadrature mainly scales with~$m$ and~$s$, but not with the matrix size~$N$. Thus, if $N$ is increased, the quadrature overhead will become negligible compared to cost of matrix-vector products and orthogonalization. Further, in the quadrature-based restarted Arnoldi method the quadrature rule \emph{needs} to be evaluated after every restart cycle (i.e. every $r=20$ matrix-vector products) in order to compute the update. This is \emph{not} necessary with sketching, where the quadrature needs to be evaluated only once for forming the final approximant. However, if error monitoring is needed, then intermediate approximants may still need to be computed.

\rev{
\subsection{Fractional graph Laplacian example}\label{subsec:fractional_laplacian}

\begin{figure}
    \centering
    \begin{minipage}{0.49\textwidth}
    \hspace*{-5mm}\includegraphics[width=1.1\textwidth]{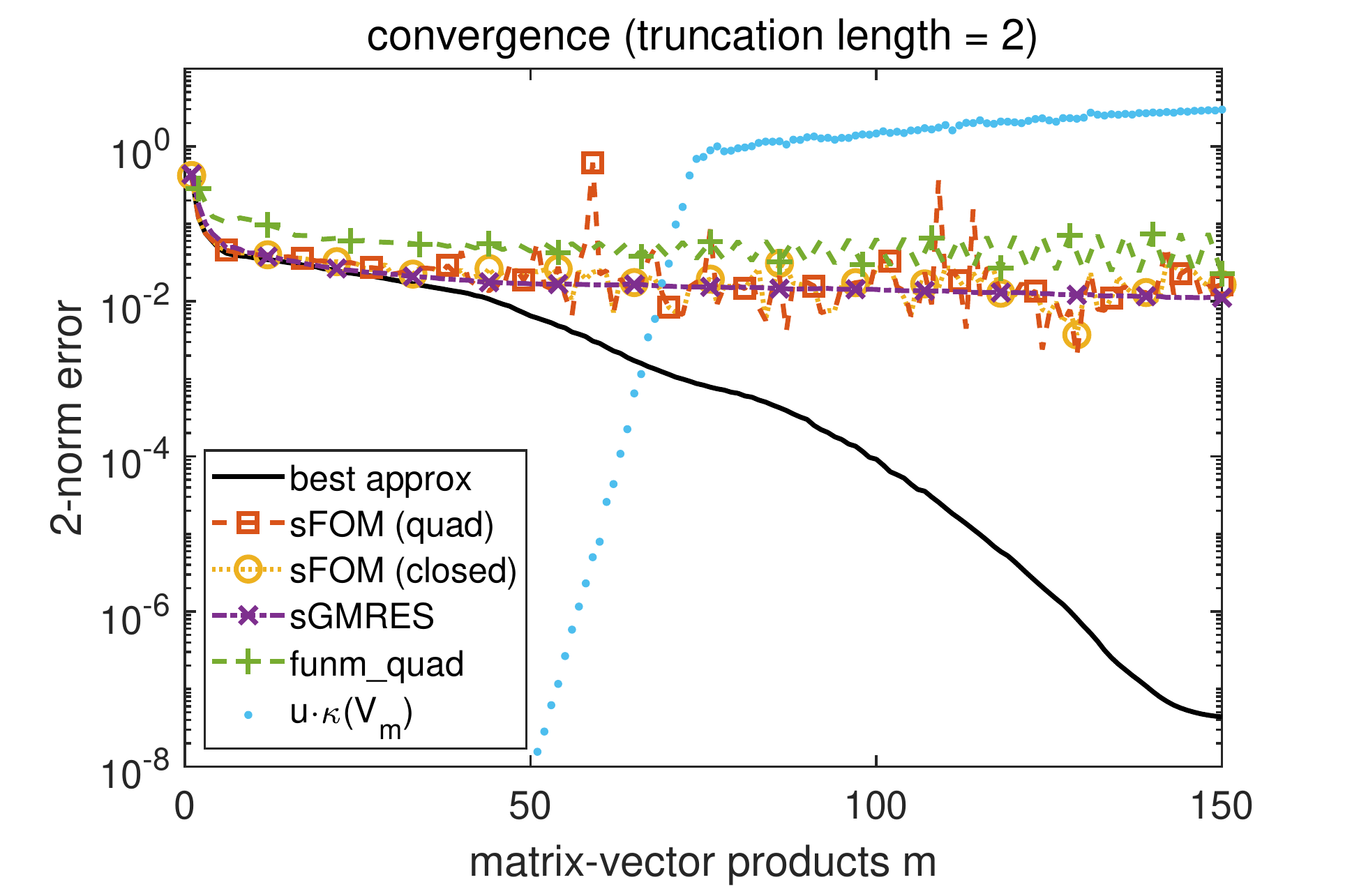}
    \end{minipage}
    \begin{minipage}{0.49\textwidth}
    \hspace*{-5mm}\includegraphics[width=1.1\textwidth]{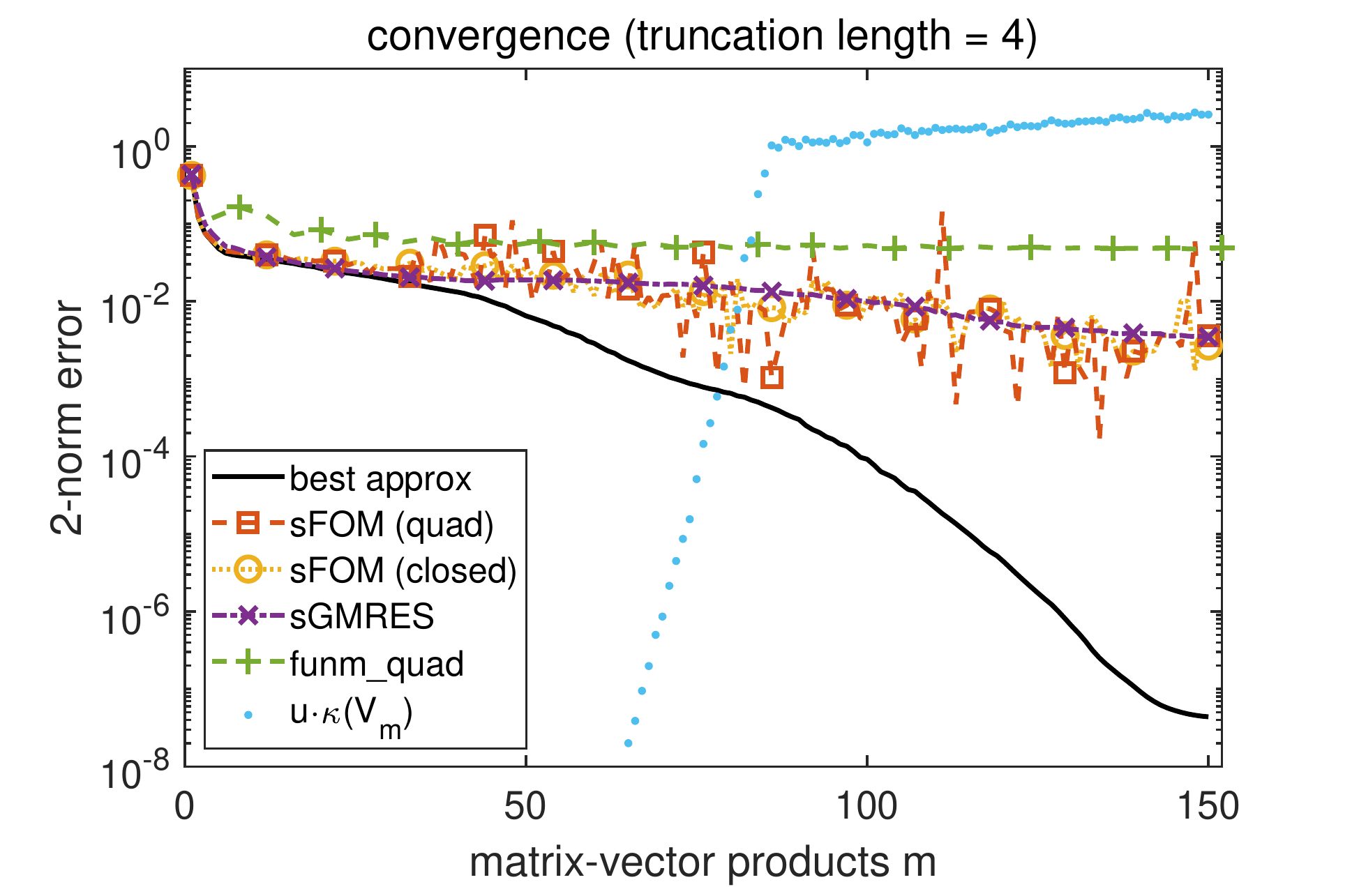}
    \end{minipage}
    \caption{Fractional graph Laplacian example. Convergence of the sketched methods based on truncated Arnoldi with truncation parameter $k=2,4$ for approximating $L^{1/2}\vb$. The error of the restarted Arnoldi approximation with restart length $r=k$ and the error of best approximation from the Krylov space is also shown, as well as the condition number of the truncated Krylov basis (multiplied by the unit round-off).}
    \label{fig:graph_laplace1}
\end{figure}

\begin{figure}
    \centering
    \begin{minipage}{0.49\textwidth}
    \hspace*{-5mm}\includegraphics[width=1.1\textwidth]{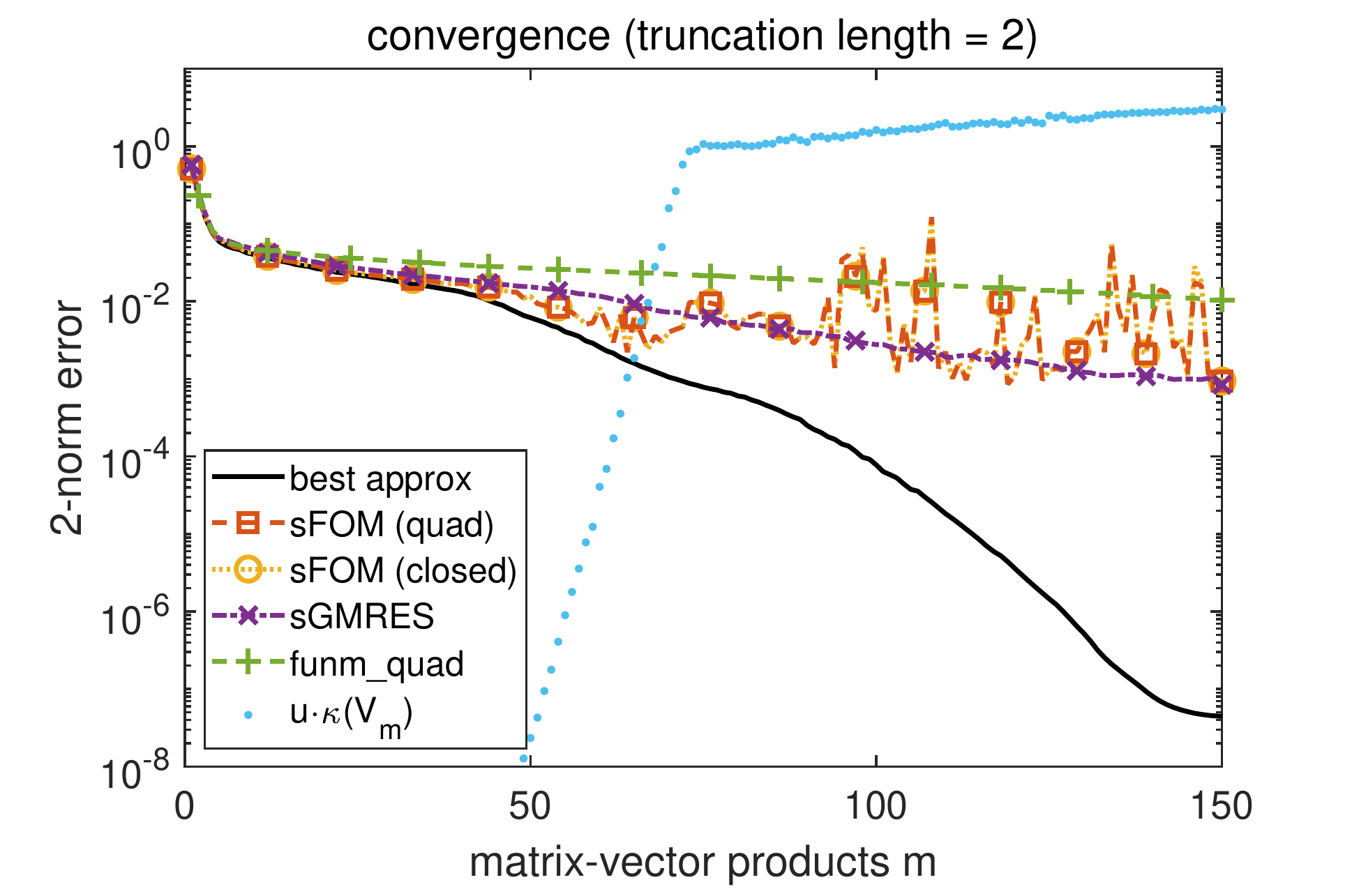}
    \end{minipage}
    \begin{minipage}{0.49\textwidth}
    \hspace*{-5mm}\includegraphics[width=1.1\textwidth]{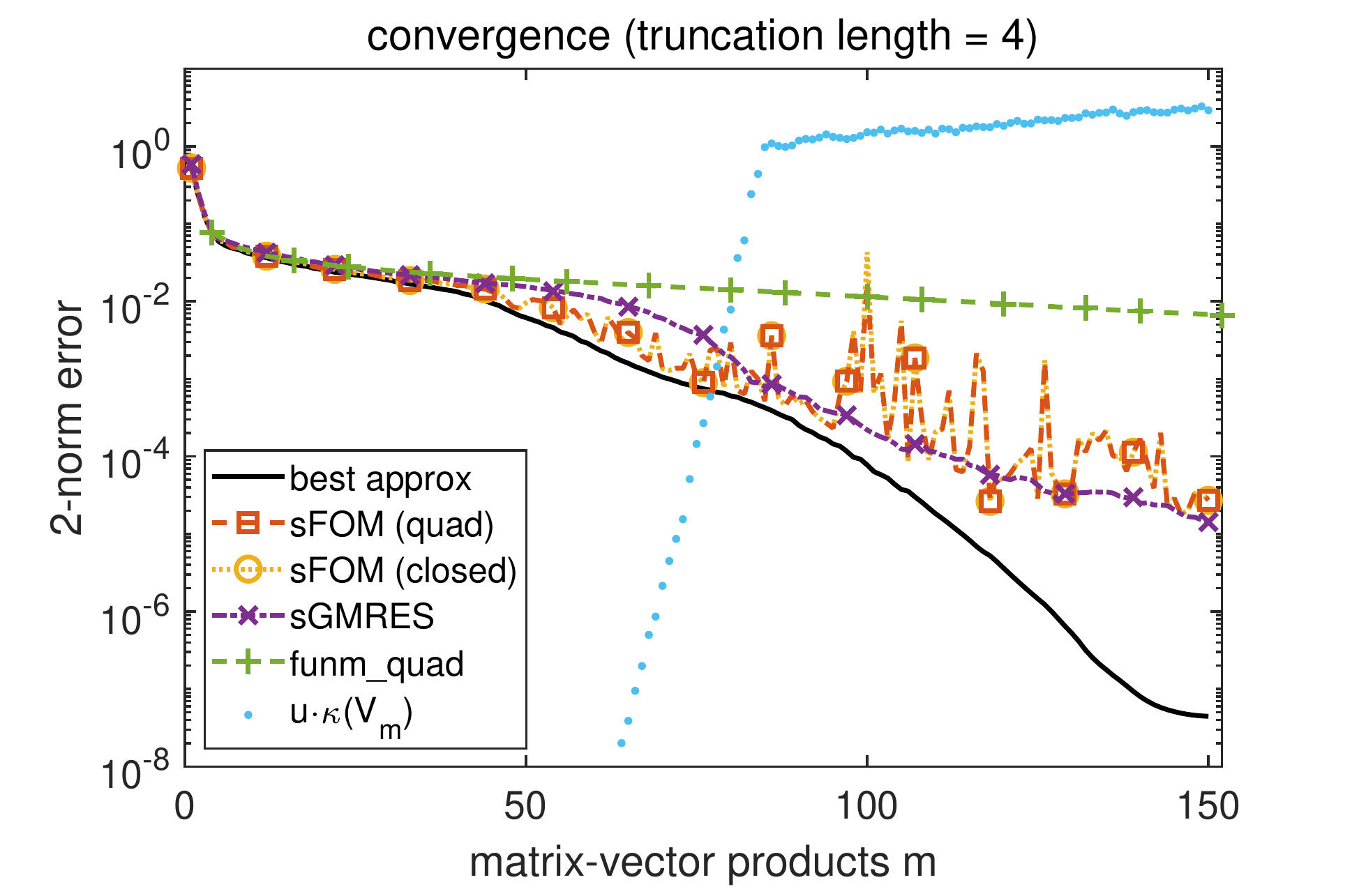}
    \end{minipage}
    \caption{Fractional graph Laplacian example. Convergence of the sketched methods based on truncated Arnoldi with truncation parameter $k=2,4$ for approximating $L^{-1/2}(L\vb)$. The error of the restarted Arnoldi approximation with restart length $r=k$ and the error of best approximation from the Krylov space is also shown, as well as the condition number of the truncated Krylov basis (multiplied by the unit round-off).}
    \label{fig:graph_laplace2}
\end{figure}

We end this section with an experiment that highlights possible problems and limitations in our approach. The example we consider here is taken from~\cite[Section~5.1.3]{cortinovis2022speeding}. We let $A \in \R^{6,301 \times 6,301}$ be the (nonsymmetric) adjacency matrix of the network \texttt{p2p\_Gnutella08} from the SuiteSparse matrix collection (\url{https://sparse.tamu.edu/}) and let $L = D_{in}-A$ be its \emph{in-degree Laplacian}, i.e., $D_{in}$ is a diagonal matrix containing the in-degree of all nodes in the network. Consequently, $L$ has zero column sums and is thus singular, with its spectrum contained in the closed right half-plane. We are interested in approximating $L^{1/2}\vb$, i.e., the action of the \emph{fractional Laplacian}, where $\vb$ is a randomly chosen canonical unit vector.

We again compare the same methods as before\footnote{\rev{Note that the published version of \texttt{funm\_quad} does not natively support the square root and we added it to the implementation using the approach outlined in~\cite[Corollary~3.6]{FrommerGuettelSchweitzer2014a}}}, with truncation (or restart) length $2$ and $4$. The results of this experiment are depicted in Figure~\ref{fig:graph_laplace1}, and it is clearly visible that all considered methods fail for this problem. While the error of the best approximation decreases smoothly (and superlinearly), the sketched and restarted methods do not converge to the solution at all or at least converge extremely slowly after a short initial phase in which they closely follow the error of the best approximation. Interestingly, deteriorating convergence begins long before the truncated Krylov basis starts becoming ill-conditioned. A reason for the unsatisfactory behavior of the sketched methods is likely the occurrence of sketched Ritz values on (or very close to) the negative real axis, i.e., the branch cut of the square root. As the origin is part of the field of values of $L$, in light of~\eqref{eq:inclusion_sketched_ritz}, these ``critical'' sketched Ritz values can occur for any $\varepsilon > 0$, i.e., irrespective of how good the subspace embedding is.

A simple trick that can be used for improving the convergence of polynomial Krylov methods for fractional graph Laplacians is to rewrite $L^{1/2}\vb = L^{-1/2}(L\vb)$ and then approximate the action of the inverse square root using a Krylov space built with starting vector $L\vb$. While $L$ is not invertible (and thus does not actually have an inverse square root), the initial multiplication $L\vb$ removes the contributions from the nullspace of $L$ from the starting vector, so that all subsequent computations happen in a space on which $L$ is an invertible operator (at least in exact arithmetic). We repeat our experiment using this approach and depict the results in Figure~\ref{fig:graph_laplace2}. We observe that convergence is indeed  greatly improved. While the methods still do not track the error of the best approximation closely, sGMRES shows at least acceptable convergence (with hints of superlinear phases) when using a truncation length of~$4$. The sketched FOM method converges at a similar rate, but shows much more irregular error norms and large spikes, likely again caused by sketched Ritz values near the negative real axis. An interesting effect is visible for sGMRES with truncation length $k = 2$. When the error of the best approximation starts to converge superlinearly, this does not happen for sGMRES, but instead convergence continues at the same linear rate as before.

In conclusion, the above example illustrates that there are cases in which the sketched methods can fail, but it also suggests that this is typically bound to happen in situations were polynomial Krylov methods are expected to show unsatisfactory performance anyway (rational Krylov methods are more commonly used for fractional graph Laplacian computations exactly for this reason). Additionally, the example (together with the preceding ones in Section~\ref{subsec:convdiff}--\ref{subsec:qcd}) highlights that ill-conditioning of the truncated Krylov basis is typically \emph{not} the primary cause of instabilities and deteriorating convergence.}

\section{Conclusions}\label{sec:concl}

We have presented several new approaches to efficiently compute Krylov approximations to $f(A)\vb$ based on integral representations and randomized sketching. We have focused on two popular Krylov methods, namely FOM and GMRES. We have shown that the sketched FOM approximant admits a closed form and provided a convergence analysis of the sketched GMRES approximants for Stieltjes function of positive real matrices. Numerical experiments have demonstrated the potential of the sketching approach as an alternative to restarting.

The proposed approach also opens up a number of research questions. \rev{Firstly, it is crucial to better understand the numerical stability (or rather, potential sources of instability) in sketched Krylov methods. While by conventional wisdom,  instabilities occur once the condition number of the non-orthogonal Krylov basis exceeds the reciprocal of unit round-off, our experiments reveal that in many cases convergence still takes place in this setting. Experience from the convergence analysis and practical use of restarted Krylov methods, even when the Arnoldi restart/truncation length is as small $k=1$, suggests that the basis conditioning cannot be the only determining factor (see also our discussion after Corollary~\ref{cor:alice}).}

Another possible research direction addresses the choice of the sketching parameter $s$. Currently, a rough estimate of the number $m_{\max}$ of required Krylov iterations is needed in order to choose~$s> m_{\max}$, and the choice $s = 2m_{\max}$ used here is not rigorously justified. One possible idea would be a ``responsibly reckless'' approach (a term used in HPC; see, e.g., \cite{dong}), where an optimistic choice for $s$ will be made initially with a careful monitoring of the computation. If $s$ turns out to be too small, which needs to be automatically detected, the computation will be halted and redone with an increased value of~$s$, or with a completely different method.

A significant improvement in the efficiency of sketched GMRES method could be obtained by developing a fast evaluation of the integral in \eqref{eq:sgmres}. In our numerical experiments with the QCD example we found that replacing MATLAB's \texttt{pinv(X)} by \texttt{(X'*X)\textbackslash X'} yields significant speed-up. The action of the Moore--Penrose inverse $(tSV_m + SAV_m)^\dagger \vv$ on a vector $\vv$ is needed for $\ell$~distinct values of $t$ corresponding to the quadrature nodes ($SV_m$ can be assumed to have orthonormal columns).  

If the dimension $m$ of the required subspace becomes extremely large, say in the order of $10^4$ and above, then it may be necessary to combine  sketching with restarting. In the restarting approach  developed in \cite{FrommerGuettelSchweitzer2014a}, each restart cycle $c=1,2,\ldots$ amounts to the Krylov approximation of an error function $e_m^{(c)}(A)\vb$, where $e_m^{(c)}(z)$ is a scalar error function that is explicitly given in terms of the  interpolation nodes defining the restarted Arnoldi approximant. \rev{In view of the interpolation characterization given in Corollary~\ref{cor:alice}, it seems plausible that a similar restarting approach might be developed for the sketching-based methods. (If $f$ is already given as a rational function in partial fraction form, then each linear system could be treated independently and the usual FOM or GMRES restarting as in \cite{AfanasjewEtAl2008a} can immediately be applied.) Restarting would also allow the use of implicit deflation techniques, which can often mitigate convergence delays.}

\section*{Acknowledgements} We are grateful for insightful discussions with \rev{Oleg Balabanov, Alice Cortinovis,} Andreas Frommer, Daniel Kressner, and Yuji Nakatsukasa. \rev{We also thank the two anonymous referees for their insightful suggestions.}

\bibliographystyle{siam}
\bibliography{refs}
\end{document}